\documentclass[12pt]{amsart} 

\usepackage[margin=3cm]{geometry}
\usepackage[utf8]{inputenc}
\usepackage[dvipsnames]{xcolor}

\usepackage{graphicx}
\usepackage{amsmath,amssymb,amsthm,mathtools}
\usepackage{latexsym}
\usepackage{appendix}
\usepackage{units}

\mathtoolsset{showonlyrefs}
\newtheorem{theorem}{Theorem}[section]
\newtheorem{corollary}[theorem]{Corollary} 
\newtheorem{lemma}[theorem]{Lemma}
\newtheorem{proposition}[theorem]{Proposition}

\newtheorem{conjecture}[theorem]{Conjecture}

\newtheorem{remark}[theorem]{Remark}

\theoremstyle{definition}
\newtheorem{example}[theorem]{Example}
\newtheorem{definition}[theorem]{Definition}

\newcommand{\Mcal}{\mathcal{M}}
\newcommand{\Acal}{\mathcal{A}}
\newcommand{\Ucal}{\mathcal{U}}
\newcommand{\Wcal}{\mathcal{W}}
\newcommand{\Ecal}{\mathcal{E}}
\newcommand{\Tcal}{\mathcal{T}}
\newcommand{\Ccal}{\mathcal{C}}

\newcommand{\al}{\alpha}
\newcommand{\ii}{\mathbf{i}}
\newcommand{\of}{\otimes5}

\newcommand{\RR}{\mathbb{R}}
\newcommand{\la}{\lambda}

\newcommand{\Lcal}{\mathcal{L}}
\newcommand{\Scal}{\mathcal{S}}
\DeclareMathOperator{\Span}{Span}
\DeclareMathOperator{\Mod}{mod}
\DeclareMathOperator{\SAdj}{SAdj}
\DeclareMathOperator{\Adjoin}{Adjoin}
\DeclareMathOperator{\rk}{rk}
\DeclareMathOperator{\srk}{srk}
\DeclareMathOperator{\drk}{drk}
\DeclareMathOperator{\sdrk}{sdrk}
\DeclareMathOperator{\Vect}{Vect}
\DeclareMathOperator{\minrk}{\min \rk}
\DeclareMathOperator{\minsrk}{\min \srk}

\title[Lower bounds on the rank and symmetric rank of real tensors]{Lower bounds on the rank and symmetric rank \\ of real tensors}

\author{Kexin Wang and Anna Seigal}
\date{}

\begin{document}

\begingroup
\let\MakeUppercase\relax % this disables uppercasing authors
\maketitle
\endgroup

\begin{abstract}
 We lower bound the rank of a tensor by a linear combination of the ranks of three of its unfoldings, using Sylvester's rank inequality. In a similar way, we lower bound the symmetric rank by a linear combination of the symmetric ranks of three unfoldings. 
Lower bounds on the rank and symmetric rank of tensors are important for finding counterexamples to Comon's conjecture. 
    A real counterexample to Comon's conjecture is a tensor whose real rank and real symmetric rank differ. Previously, only one real counterexample was known. We divide the construction into three steps. The first step involves linear spaces of binary tensors. The second step considers a linear space of larger decomposable tensors. The third step is to verify a conjecture that lower bounds the symmetric rank, on a tensor of interest. We use the construction to build an order six real tensor whose real rank and real symmetric rank differ.

\end{abstract}




\section{Introduction}

Tensors are multidimensional arrays.
We consider real tensors $\Tcal \in \RR^{I_1} \otimes \cdots \otimes \RR^{I_d}$, where $\RR^{I_j}$ is the vector space with basis elements indexed by the set $I_j$. After fixing a basis for each vector space, the tensor $\Tcal$ is a
multidimensional array of $\prod_{j=1}^d I_j$ real entries. The entry of $\Tcal$ at $(k_1, \ldots, k_d) \in I_1 \times \cdots \times I_d$ is denoted $\Tcal(k_1 | \ldots | k_d)$.
The number of indices $d$ is called the order of $\Tcal$.
Tensors 
appear in statistics~\cite{anandkumar2013tensor,anandkumar2014tensor,mccullagh2018tensor,robeva2019duality,bi2021tensors}, complexity theory~\cite{burgisser2011geometric,landsberg2017geometry}, biological data analysis~\cite{gtex2015genotype,hore2016tensor,subramanian2017next,schurch2020coordinated,ahern2021blood}, and many other applications.

A tensor~$\Tcal \in (\RR^I)^{\otimes d}$ is symmetric if its entries are unchanged under permuting indices; i.e., if $\Tcal(k_1 | \ldots | k_d) = \Tcal (\sigma(k_1) | \ldots | \sigma(k_d))$ for $\sigma$ any permutation of $d$ letters.
For example, the moment tensors of probability distributions and the higher order derivatives of smooth functions are symmetric tensors.
There is a natural correspondence between symmetric tensors in $(\RR^I)^{\otimes d}$ and homogeneous polynomials of degree $d$ in $|I|$ variables with coefficients in $\RR$. The bijection~is 
$$
\Tcal \quad \leftrightarrow \sum_{k_1,\ldots,k_d \in I} \Tcal(k_1| \ldots |k_d) x_{k_1} \! \cdots x_{k_d}.
$$
We will refer to a symmetric tensor and its corresponding polynomial interchangeably. 
In this paper, we consider tensor ranks defined over the real numbers. 
These can be be greater than those over the complex numbers, see e.g.~\cite[Example 8.3]{comon2008symmetric}.

\begin{definition}
A tensor $\Tcal$ is {\em decomposable} (or {\em has rank at most one}) if
there exist vectors $v_j \in \RR^{I_j}$ for all $j \in \{ 1 , \ldots,  d\}$ such that 
$$\Tcal(k_1 | \ldots | k_d) = v_1 (k_1) \cdots v_d(k_d) .$$
The {\em rank} $\rk \Tcal$ is the minimal $r$ such that~$\Tcal$ can be written as the sum of $r$ decomposable tensors. 
For symmetric $\Tcal$, the {\em symmetric rank} $\srk \Tcal$ is the minimal $r$ such that $\Tcal$ can be written as the sum of $r$ symmetric decomposable tensors. 
\end{definition}

Writing a tensor as a sum of rank one terms decomposes it into building blocks that can be interpreted in a context of interest, such as recovering parameters in a mixture model~\cite{lim2009nonnegative,anandkumar2014tensor,sullivant2018algebraic} and counting 
the multiplications in an optimal algorithm for a linear operator~\cite{landsberg2017geometry}. The symmetric rank appears in independent component analysis while the rank arises in multiway factor analysis~\cite{comon2008symmetric}.

There are many numerical algorithms to decompose a tensor~\cite{vervliet2016tensorlab,kolda2006matlab}. However, there are few exact tools and it is difficult to find the exact rank or symmetric rank of a tensor~\cite{haastad1989tensor,hillar2013most,landsberg2012tensors}. The main challenge is to find lower bounds for the rank, since an upper bound is obtained by exhibiting a decomposition. Known methods to lower bound the rank of a tensor that apply in general are the substitution method~\cite{burgisser2013algebraic}, lower bounding by the rank of a flattening or unfolding~\cite{landsberg2012tensors}, and using the singularities of a hypersurface defined by the tensor~\cite{landsberg2010ranks}.

The rank and symmetric rank coincide for a symmetric matrix; i.e., for an order two tensor. The rank can be found from a matrix decomposition such as the eigendecomposition and singular value decomposition. 
The question of whether the rank and symmetric rank are always equal for higher order tensors was posed by Comon.
First results for the agreement of rank and symmetric rank were given in~\cite{comon2008symmetric}. 
The assertion that the rank and symmetric rank of a tensor always agree is known as Comon's conjecture.
There has been significant progress into Comon's conjecture, see e.g.~\cite{friedland2016remarks,zhang2016comon}.
The conjecture has also been posed for tensors over other fields~\cite{zheng2020comon}, for partially symmetric decompositions~\cite{gesmundo2019partially}, and for the border rank of a tensor~\cite{buczynski2013determinantal}, which may differ from the rank~\cite{de2008tensor}.

However, Comon's conjecture was disproved via construction of a complex counterexample~\cite{shitov2018counterexample} and a real counterexample~\cite{shitov2020comon}. These two counterexamples demonstrate how linear algebra along the different indices of a tensor can combine in unintuitive ways.
The paper~\cite{shitov2018counterexample} constructs a symmetric $800 \times 800 \times 800$ tensor with complex rank $903$ and complex symmetric rank at least $904$. The paper~\cite{shitov2020comon} shows the existence of a real symmetric tensor of format $208 \times 208 \times 208 \times 208$, with rank $761$ and symmetric rank $762$. 
To date, these large tensors are the only known counterexamples.
In comparison, the agreement of rank and symmetric rank was shown for small tensors in~\cite{seigal2019structured,seigal2020ranks}. 
The problem of finding a minimal size, or minimal rank, counterexample to Comon's conjeture remains unsolved. The border rank analogue to the conjecture also remains open. 

\bigskip

In this paper, our first main contribution is to give new lower bounds on the rank and symmetric rank of a tensor. To state the lower bounds, we first recall the standard notions of flattenings and slices of a tensor.

\begin{definition}[Flattenings]
\label{def:flatten}
Fix $\Tcal \in \mathbb{R}^{I_1}\otimes\cdots\otimes\mathbb{R}^{I_d}$ and a subset $J \subset [d]$. 
The $J$ {\em flattening} of $\Tcal$, denoted $\Tcal^{(J)}$, is a matrix with rows indexed by $\times_{j \in J} I_j$ and columns indexed by $\times_{h \notin J} I_h$. The entry of $\Tcal^{(J)}$ at row index $(k_j : j \in J)$ and column index $(k_h : h \notin J)$ is  $\Tcal (k_1 | \ldots | k_d)$.
For $J = \emptyset$ we obtain a vector $\Tcal^{(\emptyset)} \in \RR^{\prod_{j=1}^d I_j}$. We call this vector the {\em vectorisation} of $\Tcal$ and denote it by $\Vect \Tcal$.
\end{definition}

A partition $[d] = J_1 \cup \cdots \cup J_\delta$ gives an order~$\delta$ {\em unfolding} of $\Tcal$, whose entry at $((k_j : j \in J_1), \ldots, (k_j : j \in J_\delta))$ is $\Tcal (k_1 | \ldots | k_d)$. The $J$ flattening is the case $[d] = J \cup J^c$.

\begin{definition}[Slices]
\label{def:slices}
Given $\Tcal \in \RR^{I_1} \otimes \cdots \otimes \RR^{I_d}$, its 
$i$th $j$ {\em slice} $\Tcal^j_i \in \RR^{I_1} \otimes \cdots \otimes \RR^{I_{j-1}} \otimes \RR^{I_{j+1}} \otimes \cdots \otimes \RR^{I_d}$ is obtained by fixing the $j$th index of $\Tcal$ to take value $i$,
$$\Tcal^j_i(k_1|k_2|\ldots|k_{j-1}|k_{j+1}|\ldots|k_{d-1}|k_d)=\Tcal(k_1|k_2|\ldots|k_{j-1}|i|k_{j+1}|\ldots|k_{d-1}|k_d).
$$
Fixing $\ii = (\ii_j : j \in J) \in \times_{j \in J} I_j$ for $J \subset [d]$, the $\ii$th $J$ slice $\Tcal_\ii^J \in \otimes_{h \notin J} \RR^{I_h}$ is obtained by fixing index $j$ to take value $\ii_j$, for all $j \in J$.
\end{definition}

The columns of the flattening $\Tcal^{(J)}$ are the vectorisations  of the slices $\Tcal_\ii^{J^c}$, where $J^c = [d] \backslash J$ and $\ii$ ranges over $\times_{h \notin J} I_h$.

\bigskip

To state our first main contribution, we give the following new definitions. 

\begin{definition}
\label{def:Lj}
The $J$th {\em slice space} $\Lcal_J \subset \otimes_{j \in J} \RR^{I_j}$ is the span of $\{ \Tcal_\ii^{J^c} : \ii \in \times_{h \notin J} I_h \}$; i.e., the span of the tensors whose vectorisations appear as the columns of $\Tcal^{(J)}$. 
\end{definition}

\begin{definition}
\label{def:drk}
The {\em $J$th decomposable flattening rank} of $\Tcal$, denoted $\drk_J \Tcal$, is the smallest $r$ such that there exist $r$ decomposable tensors in $\otimes_{j \in J}\RR^{I_j}$ whose linear span contains the slice space $\Lcal_J$.
\end{definition}

We note the comparison with decompositions to compute the strength of a tensor~\cite{bik2019polynomials}, which depend on indexing sets that may vary from one summand to the next.

For a symmetric tensor $\Tcal \in (\RR^I)^{\otimes d}$, the flattening $\Tcal^{(J)}$ only depends on $J$ via $j = |J|$, so we abbreviate $\Tcal^{(J)}$ to $\Tcal^{(j)}$.
Similarly, we abbreviate $\Lcal_J$ to $\Lcal_j$ and $\drk_J \Tcal$ to $\drk_j \Tcal$.

\begin{definition}
\label{def:sdrk}
The {\em $j$th symmetric decomposable flattening rank} of $\Tcal^{(j)}$, denoted $\sdrk_j \Tcal$, is the smallest $r$ such that there exist $r$ symmetric decomposable tensors in $(\RR^I)^{\otimes j}$ that span the slice space $\Lcal_j$. \end{definition}

\begin{remark}
Definition~\ref{def:sdrk}, with $\mathbb{C}$ instead of $\RR$,
is the $j$th gradient rank from~\cite[Definition 1.2]{gesmundo2019partially}.
However, Definition~\ref{def:sdrk} differs from the decomposable symmetric rank in~\cite{rodriguez2021rank}, the smallest $r$ such that a symmetric tensor can be written as the sum of $r$ tensors of the form $\frac{1}{d!}\sum_{\sigma\in S_d}z_{\sigma(1)}\otimes\cdots\otimes z_{\sigma(d)}$.
\end{remark}

Our first main result is the following lower bounds on the rank and symmetric rank.

\begin{theorem}\label{thm:A}
Let $\Tcal$ be an order $d$ tensor, and fix $J \subset [d]$, with $J^c := [d] \backslash J$ and $j = |J|$. Then
$$ \rk \Tcal \, \geq \, \drk_J \Tcal + \drk_{J^c} \Tcal - \rk \Tcal^{(J)} . $$
If $\Tcal$ is symmetric then
$$ \srk \Tcal \geq \sdrk_j \Tcal + \sdrk_{d-j} \Tcal - \rk \Tcal^{(j)} . $$
\end{theorem}

Theorem~\ref{thm:A} gives a tight lower bound on the rank of the quaternary quartic polynomial (or, symmetric $4 \times 4 \times 4 \times 4$ tensor) 
\begin{equation}
    \label{eqn:quartic_ex}
    x^4-3y^4+ 12x^2yz+12xy^2w,
\end{equation}
see Corollary \ref{cor:geq12} and Proposition \ref{prop:eq12}.
The coefficients ensure integer entries in the tensor.
This polynomial is the starting point to the construction of a real counterexample to Comon's conjecture from~\cite{shitov2020comon}. A tight lower bound is not possible via the substitution method, by lower bounding by the rank of a single unfolding, or using the lower bound in~\cite{landsberg2010ranks}.

The paper~\cite{shitov2020comon} constructs an order four counterexample to Comon's conjecture. The paper also gives a framework for the construction of counterexamples to Comon's conjecture. We make a small simplification, removing the need for two conditions.
 We break down the construction into three steps. 
The last step is to prove a conjecture to lower bound the real symmetric rank of a tensor of interest. This conjecture (Conjecture~\ref{conj:modspan}) is the real analogue to~\cite[Conjecture 6]{shitov2018counterexample}.
Proving Conjecture~\ref{conj:modspan} would give a clearer path to finding more counterexamples to Comon's conjecture.
The paper~\cite{shitov2020comon} states that the construction potentially allows one to construct counterexamples for tensors of any even order $d \geq 4$.
Our second main result is to resolve the next case $d = 6$ using combinatorial and linear algebraic arguments. 

\begin{theorem}
\label{thm:B}
There is an order six real tensor whose rank and symmetric rank differ.
\end{theorem}

The rest of this paper is organised as follows. We outline preliminaries in Section~\ref{sec:prelim}. We prove Theorem~\ref{thm:A} in Section~\ref{sec:decomposable_flattening}, where we also state Conjecture~\ref{conj:modspan} and use Theorem~\ref{thm:A} to prove it in special cases. In Section~\ref{sec:idea} we describe three steps to construct a counterexample to Comon's conjecture, extracted from~\cite{shitov2020comon}. We construct an order six counterexample in Section~\ref{sec:order6}, with some proofs given in~\ref{sec:appendix}. We conclude with some open problems.

\section{Preliminaries}
\label{sec:prelim}

For background on tensors see~\cite{landsberg2012tensors} and~\cite{hackbusch2012tensor}.
Recall the definitions of flattenings and slices from Definitions~\ref{def:flatten} and~\ref{def:slices}.

\begin{theorem}[{The real substitution method, see~\cite[Lemma B.1]{alexeev2011tensor},~\cite[Theorem 4.4]{seigal2020ranks},~\cite[Lemma 4.6]{shitov2020comon}}]
\label{thm:substitution}
Fix $\Tcal \in  \mathbb{R}^{I_1}\otimes\cdots\otimes\mathbb{R}^{I_d}$ with $j$ slices $\Tcal^j_1,\cdots,\Tcal^j_{n}$, where $I_j = [n]$.
There exist $c_1,\ldots,c_{n-1} 
\in \RR$ such that 
$$
\rk\Tcal \, \geq \, \rk(\Tcal_1^j+c_1\Tcal_{n}^j|\cdots|\Tcal_{n-1}^j+c_{n-1}\Tcal_{n}^j)+1.
$$
Equality holds if the slice $\Tcal_{n}^j$ is decomposable.
\end{theorem}

Following~\cite[Section 4]{shitov2020comon}, we define some linear operations on tensors. 
We keep most notation consistent with~\cite{shitov2020comon}.
Fix $\Ccal \in  \mathbb{R}^{I_1}\otimes\cdots\otimes\mathbb{R}^{I_d}$ and consider $d$ finite sets of order $(d-1)$ tensors
$$ \Mcal_j \subset\mathbb{R}^{I_1}\otimes\cdots \otimes \mathbb{R}^{I_{j-1}}\otimes\mathbb{R}^{I_{j+1}}\otimes\cdots\otimes\mathbb{R}^{I_d},  \qquad j \in \{ 1, \ldots, d \}.$$
We index the tensors in $\Mcal_j$ by the set $W_j$.

\begin{definition}[{Adjoining slices to a tensor, see~\cite[Definitions 4.7 and 4.8]{shitov2020comon}}]
\label{def:adj}
The {\em adjoining} of $\Mcal_1, \ldots, \Mcal_d$ to $\Ccal$ is the tensor 
$$ \Tcal: =\Adjoin(\mathcal{C},\Mcal_1,\ldots,\Mcal_d) \in  \mathbb{R}^{I_1\cup W_1}\otimes\cdots \otimes \mathbb{R}^{I_d\cup W_d} , $$
with entries
\begin{enumerate}
    \item $\Tcal(k_1|\ldots|k_d)=\Ccal(k_1|\ldots|k_d)$ if $k_j\in I_j$ for all $j\in\{1,\ldots,d\}$
    \item 
    $\Tcal(k_1|\ldots|k_{j-1}| w| k_{j+1}|\ldots | k_d)=\Mcal_j^{(w)}(k_1|\ldots|k_{j-1}|k_{j+1}|\ldots|k_d)$ if $k_h \in I_h$ for all $h \neq j$ and $w \in W_j$, 
where $\Mcal_j^{(w)}$ is the tensor in $\Mcal_j$ indexed by $w$. 
\item $\Tcal(k_1|\ldots|k_d) = 0$ otherwise, i.e. if $k_j \notin I_j$, for more than one $j \in \{ 1 , \ldots, d\}$.
\end{enumerate}
See Figure~\ref{fig:adjoing} for an illustration. If $I:= I_1=\ldots=I_d$ and $\Mcal := \Mcal_1=\ldots=\Mcal_d$ is a finite set of tensors indexed by $W$, the {\em symmetric adjoining} of $\Mcal$ to $\Ccal$ is
$$
{\SAdj}(\Ccal,\Mcal):=\Adjoin(\Ccal,\Mcal,\ldots,\Mcal) \in (\RR^{I \cup W})^{\otimes d}.
$$
\end{definition}

\begin{figure}
    \centering
    \includegraphics[width=5cm]{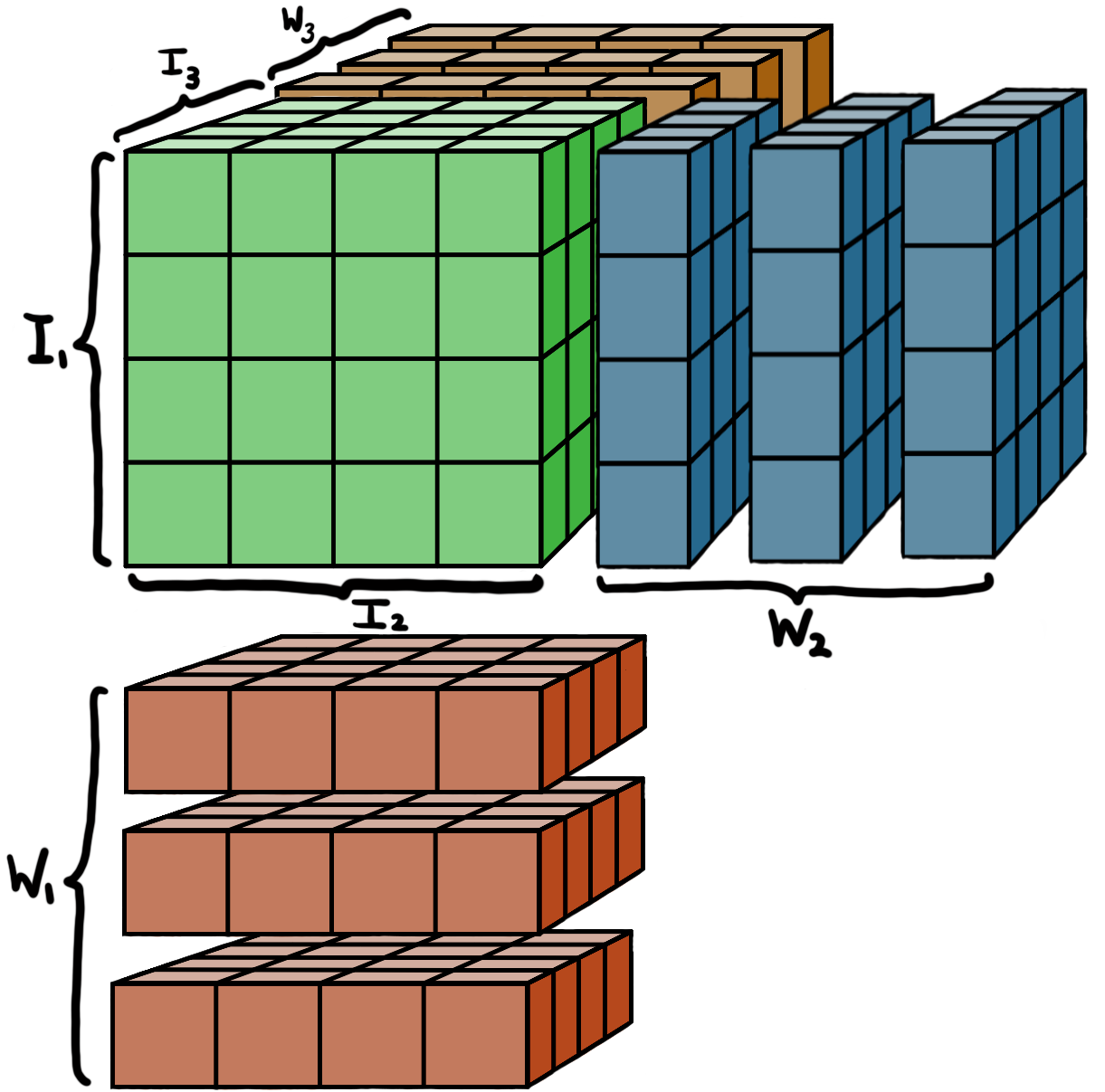}
    \caption{Illustration of Definition~\ref{def:adj}. Three sets $\Mcal_j$ are adjoined to a tensor in $\RR^{I_1} \otimes \RR^{I_2} \otimes \RR^{I_3}$ to produce a tensor in $\RR^{I_1 \cup W_1} \otimes \RR^{I_2 \cup W_2} \otimes \RR^{I_3 \cup W_3}$.}
    \label{fig:adjoing}
\end{figure}

\begin{definition}[{\cite[Definition 4.4]{shitov2020comon}}]
The set $\Ccal \Mod (\Mcal_1,\ldots,\Mcal_d)$ is the linear space
of tensors obtained from~$\Ccal$ by adding an element of $\Span \Mcal_j$ to every $j$ slice of~$\Ccal$, for all $j \in \{1, \ldots, d\}$. 
That is, $\Ccal \Mod (\Mcal_1,\ldots,\Mcal_d)$ is the space of tensors with entries
$$ \Ccal(k_1 | \ldots | k_d) + M_1^{(k_1)} ( k_2 | \ldots | k_d) + M_2^{(k_2)} ( k_1 | k_3 | \ldots | k_d) + \cdots + M_d^{(k_d)} ( k_1 | \ldots | k_{d-1} ) ,$$
where $M_j^{(k_j)} \in \Span \Mcal_j$ for all $j \in \{ 1, \ldots, d \}$ and all $k_j \in I_j$. 
If $I_1=\ldots=I_d$ and $\Mcal := \Mcal_1=\ldots=\Mcal_d$, 
we define $\mathcal{C} \Mod \mathcal{M} := \mathcal{C} \Mod (\mathcal{M},\ldots,\mathcal{M})$.
The elements of $\Span \Mcal$ added to each of the $j$ slices of $\Ccal$ need not be the same.
\end{definition}

Given a set of tensors $\mathcal{A}$, define $\minrk \mathcal{A}$ to be the minimal rank of a tensor in $\mathcal{A}$. We have the following consequence of Theorem~\ref{thm:substitution}.

\begin{corollary}[{\cite[Lemma 4.10]{shitov2020comon}}]
\label{cor:adoined_substitution}
Fix $\Ccal$ and $\Mcal_1, \ldots, \Mcal_d$ as above. Then
$$
\rk \Adjoin (\mathcal{C},\Mcal_1,\ldots,\Mcal_d) \geq \minrk \left( \mathcal{C} \Mod (\Mcal_1,\ldots,\Mcal_d) \right) +\sum_{j=1}^{d} \dim \Span \Mcal_j.
$$
Equality holds if each linear space $\Span \Mcal_j$ has a basis of  decomposable tensors.
\end{corollary}

\section{Lower bounds on the rank and symmetric rank} 
\label{sec:decomposable_flattening}

In this section, we study the decomposable flattening rank and symmetric decomposable flattening rank, from Definitions~\ref{def:drk} and~\ref{def:sdrk}. We combine the notions with Sylvester's rank inequality to prove Theorem~\ref{thm:A}. This result enables us to find the rank of a tensor by 
studying the decomposable matrices in a certain linear space. We see in examples that our new lower bounds can improve on existing lower bounds.
We discuss a symmetric analogue to the real substitution method in Conjecture~\ref{conj:modspan}.
Although we focus on real ranks, much of what we discuss extends to complex ranks.

\subsection{Decomposable flattening rank}
%% up to here 

Recall the definitions of flattenings from Definition~\ref{def:flatten}, the slice space from Definitions~\ref{def:Lj}, and the decomposable flattening rank from Definition~\ref{def:drk}.

\begin{example}\label{exmp:L_2}
Let $\Tcal = x^3y$. Then $\rk \Tcal=4$~\cite[Proposition 5.6]{comon2008symmetric}. We have
$$
\Tcal^{(2)}=
\begin{pmatrix}
0&1&1&0\\
1&0&0&0\\
1&0&0&0\\
0&0&0&0\\
\end{pmatrix}
\quad \text{and} \quad 
\Lcal_2 = \left\{ \begin{pmatrix} a_1 & a_2 \\ a_2 & 0 \end{pmatrix} \mid a_1, a_2 \in \RR \right\}.
$$
Observe that $\rk \Tcal^{(2)}=2$. Moreover, $\drk_2 \Tcal \in \{ 2, 3\}$, since the space of $2 \times 2$ symmetric matrices has dimension~$3$. Assume $\drk_2 \Tcal = 2$, for contradiction. Then $\Lcal_2 \subseteq \langle M_1, M_2 \rangle$ for some decomposable $M_1, M_2 \in \RR^{2 \times 2}$. Since $\Lcal_2$ is two-dimensional, this containment is an equality, hence $M_1, M_2 \in \Lcal_2$. But any decomposable matrix in $\Lcal_2$ has $a_2=0$, hence $M_1$ and $M_2$ are collinear, a contradiction. Hence $\drk_2 \Tcal =3$. For this example, $\rk \Tcal^{(2)} < \drk_2 \Tcal < \rk \Tcal$. 
\end{example}

\begin{proposition}
\label{prop:orderJ+1}
The decomposable flattening rank $\drk_J \Tcal$ is the rank of the order $|J| + 1$ unfolding of $\Tcal$ whose $|J|+1$ slices are $\{ \Tcal_\ii^{(J^c)} \mid \ii \in \times_{h \notin J} I_h \}$.
\end{proposition}

\begin{proof}
Denote the order $|J|+1$ unfolding by $\Scal$. Let $\{\Ucal_1,\ldots,\Ucal_{r}\}$ be decomposable tensors whose span contains $\Lcal_{J}$, where $r=\drk_J \Tcal$. Each $\Tcal_\ii^{(J^c)}$ can then be written as a linear combination of $\Ucal_1,\ldots,\Ucal_{r}$, say $\Tcal_{\ii}^{(J^c)}=\sum_{k=1}^r c_{\ii}^{(k)}\Ucal_k$. These linear combinations combine to give an expression for $\Scal$ as a sum of $r$ decomposable tensors
\begin{equation}
    \label{eqn:Stensor}
    \Scal=\sum_{\ii\in\times_{h \notin J} I_h} \sum_{k=1}^r c_{\ii}^{(k)}\Ucal_k\otimes e_{\ii} = \sum_{k=1}^r \Ucal_k\otimes \left( \sum_{\ii\in\times_{h \notin J} I_h}c_{\ii}^{(h)}e_{\ii} \right).
\end{equation}
Hence $\rk \Scal \leq \drk_J \Tcal$.
Conversely, if $\Scal$ is the sum of $r'$ decomposable tensors $\{ x_i^{(1)}\otimes\ldots\otimes x_i^{(|J|+1)} \mid i\in\{1,\ldots,r' \}\}$, then each $|J|+1$ slice of $\Scal$ lies in  $\langle x_i^{(1)}\otimes\ldots\otimes x_i^{(|J|)} \mid i\in\{1,\ldots,r'\} \rangle$, hence $\rk\Scal \geq \drk_J \Tcal.$ In conclusion, $\drk_J \Tcal = \rk \Scal$.
\end{proof}

\begin{proposition}\label{prop:drkineq}
Fix $\Tcal \in \RR^{I_1} \otimes \cdots \otimes \RR^{I_d}$ and $J \subset [d]$. Then
\begin{itemize}
\item[(i)] We have $\rk \Tcal^{(J)} \leq \drk_J \Tcal \leq \rk \Tcal$,
\item[(ii)] If $|J| = 1$ then $\drk_J \Tcal = \rk \Tcal^{(J)}$,
\item[(iii)] If $|J| = d-1$ then $\drk_J \Tcal = \rk \Tcal$, and
\item[(iv)] If $J' \subset J \subset [d]$ then $\drk_{J'} \Tcal \leq \drk_J \Tcal$.
\end{itemize}
\end{proposition}

\begin{proof}
Any decomposition of a tensor gives a decomposition of its unfoldings. Statements (i)-(iv) then follow from Proposition~\ref{prop:orderJ+1}.
\end{proof}

The inequalities in Proposition \ref{prop:drkineq} can be strict, see Example~\ref{exmp:L_2} and the following. 

\begin{example}
\label{exmp:L_3}
Set $\Tcal = x^4 y$, $J' = \{1, 2\}$ and $J = \{ 1,2,3\}$. Then $\drk_{J'} \Tcal =3$ and $\drk_J \Tcal =4$, as follows. The slice spaces are
$$
\Lcal_{J'}=\langle xy, x^2 \rangle \quad \text{and} 
\quad
\Lcal_{J}=\langle x^2y, x^3 \rangle.
$$
The slice space $\Lcal_{J'}$ appeared in Example \ref{exmp:L_2}, so $\drk_{J'} \Tcal =3$. We have $\drk_J \Tcal \in \{3,4\}$, since $x^2y$ has rank $3$ and $x^3$ has rank $1$. But any rank 3 decomposition of $x^2y$ does not contain $x^3$ in its span, see Lemma \ref{lem:decomp_mu}, so $\drk_J \Tcal =4$. \end{example}

The decomposable flattening rank can be studied via the ideal of decomposable tensors in a linear space. 
This gives lower bounds on the difference $\drk_J \Tcal - \rk \Tcal^{(J)}$.
We saw this idea in Example~\ref{exmp:L_2}. 
We illustrate the approach on~\eqref{eqn:quartic_ex}, a larger example.

\begin{proposition}
\label{prop:drk9}
Fix $\Tcal = x^4-3y^4+ 12x^2yz+12xy^2w$. Then $\drk_2 \Tcal =\sdrk_2 \Tcal= 9$.
\end{proposition}

\begin{proof}
The slice space is $\Lcal_2 = \langle x^2,xy,y^2,xz-yw,yz,xw \rangle$ or, in coordinates,
\begin{equation}
    \label{eqn:4x4}
    \Lcal_2 = \left\langle 
    \begin{pmatrix}
a_1 & a_2 & a_4 & a_6 \\
a_2 & a_3 & a_5 & -a_4 \\
a_4 & a_5 & 0 & 0 \\
a_6 & -a_4 & 0 & 0
\end{pmatrix} \mid a_1, \ldots, a_6 \in \RR \right\rangle \subset \RR^{4 \times 4}. 
\end{equation} 
The nine symmetric decomposable matrices $x^2$, $(x+y)^2$, $y^2$, $(x+z)^2$, $(x+w)^2$, $(y+z)^2$, $(y+w)^2$, $z^2$, $w^2$ span $\Lcal_2$, hence  $\sdrk_2 \Tcal \leq 9$.
It remains to show that $\drk_2 \Tcal \geq 9$. 

The decomposable rank $\drk_2 \Tcal$ is the smallest $r$ such that $r$ rank one $4 \times 4$ matrices span $\Lcal_2$.
Let $\mathcal{K}$ denote the span of these rank one matrices. 
If $\dim \mathcal{K} = 6$, then every matrix in $\mathcal{K}$ is also in $\Lcal_2$. But
a decomposable matrix in $\Lcal_2$ has $a_4 = a_5 = a_6 = 0$,
and such matrices do not span $\Lcal_2$. Hence $\dim \mathcal{K} > 6$.

We extend this argument to show that $\dim \mathcal{K} > 8$. If $\dim \mathcal{K} \leq 8$, then $\mathcal{K}$ is spanned by $\Lcal_2$ together with two other rank one matrices. Then every element of $\mathcal{K}$ is
\begin{equation}
    \label{eqn:a1..a8} 
    \begin{pmatrix}
a_1 & a_2 & a_4 & a_6 \\
a_2 & a_3 & a_5 & -a_4 \\
a_4 & a_5 & 0 & 0 \\
a_6 & -a_4 & 0 & 0
\end{pmatrix} + 
a_7 
\begin{pmatrix} x_{11} \\ x_{12} 
\\ x_{13} \\ x_{14} \end{pmatrix} \otimes \begin{pmatrix} x_{21} \\ x_{22} 
\\ x_{23} \\ x_{24} \end{pmatrix}+
a_8
\begin{pmatrix} y_{11} \\ y_{12} 
\\ y_{13} \\ y_{14} \end{pmatrix} \otimes \begin{pmatrix} y_{21} \\ y_{22} 
\\ y_{23} \\ y_{24} \end{pmatrix}
\end{equation}
for fixed $x_{11}, \ldots, x_{24}$ and variable coefficients $a_1, \ldots, a_8$.
Consider the decomposable matrices of the form~\eqref{eqn:a1..a8}. 
The ideal of $2 \times 2$ minors contains
$$ a_7 a_8 (x_{14} y_{13} - x_{13} y_{14})(x_{24} y_{23} - x_{23} y_{24}) .$$
If $x_{24} y_{23} - x_{23} y_{24} = 0$, then the lower-right $2 \times 2$ block
of any matrix in $\mathcal{K}$ has
both rows proportional to $\begin{pmatrix} x_{23} & x_{24}  \end{pmatrix}$.
A decomposable matrix in $\mathcal{K}$ therefore has top right $2\times2$ block with rows proportional to $\begin{pmatrix} x_{23} & x_{24} \end{pmatrix}$. But $\mathcal{K}$ contains $\Lcal_2$, which contains matrices with rank two top right $2 \times 2$ block (e.g. $a_5 = a_6 = 1$, all other $a_i = 0$). 
This is a contradiction to $x_{24} y_{23} - x_{23} y_{24} = 0$. 
 By symmetry, this argument also excludes
 $x_{14} y_{13} - x_{13} y_{14} = 0$.
Hence $a_7 a_8 = 0$. This argument also shows that we need at least two extra matrices to span $\mathcal{K}$, hence $\dim \mathcal{K} > 7$.

We now consider decomposable matrices as in~\eqref{eqn:a1..a8} with $a_7 \neq 0$ and $a_8 = 0$. Then
\begin{equation}
\label{proportionalrel2}
a_4 x_{14}=a_6 x_{13}, \quad
a_5 x_{14}=-a_4 x_{13},\quad a_4 x_{24}=a_6 x_{23},\quad a_5 x_{24}=-a_4 x_{23}.
\end{equation} 
Hence $a_4 = a_5 = a_6 = 0$ or $a_4 a_5 a_6 \neq 0$. If $a_4 = a_5 = a_6 = 0$, then the matrix must be
\begin{equation}
\label{eqn:extra basis}
a_7 
\begin{pmatrix} x_{11} \\ x_{12} 
\\ x_{13} \\ x_{14} \end{pmatrix} \otimes \begin{pmatrix} x_{21} \\ x_{22} 
\\ x_{23} \\ x_{24} \end{pmatrix}.
\end{equation}
If $a_4 a_5 a_6 \neq 0$, 
the vectors $\begin{pmatrix}
x_{13}&x_{14}
\end{pmatrix}$ and $\begin{pmatrix}
x_{23}&x_{24}
\end{pmatrix}$ are linearly dependent, by~\eqref{proportionalrel2}. Rescaling one of the $x$ vectors, we may assume
$x_{13}=x_{23}$ and $x_{14}=x_{24}
$. Moreover,~\eqref{proportionalrel2} allows us to set $a_4 = x_{13}\alpha$, $a_6 = x_{14}\alpha$, and $a_5 = -\frac{x_{13}^2\alpha}{x_{14}}$, for some $\alpha \neq 0$. Without loss of generality, we set $a_7 = 1$.
Then the matrix is
\begin{equation}
\label{eqn:decomp a_7neq0}
\begin{pmatrix}
\alpha+x_{11}\\-\frac{x_{13}^2}{x_{14}}\alpha+x_{12}\\x_{13}\\x_{14}
\end{pmatrix}\otimes\begin{pmatrix}
\alpha+x_{21}\\-\frac{x_{13}^2\alpha}{x_{14}}+x_{22}\\x_{13}\\x_{14}
\end{pmatrix}=M_0 + \alpha M_1+ \alpha^2 M_2
\end{equation}
where $M_i$ are fixed matrices with entries given in terms of $x_{11},\ldots,x_{24}$.
The span of such matrices has dimension at most $3$. 
Moreover, the matrix $M_0$ is of the form \eqref{eqn:extra basis}. Hence, after combining with the case $a_4a_5a_6=0$, we still have a space of matrices of dimension at most $3$. Similarly, the span of the space of decomposable matrices with $a_7 = 0$ and $a_8 \neq 0$ has dimension at most~$3$.
Denote the linear spaces spanned by the decomposable matrices in~\eqref{eqn:a1..a8}  with $(a_7 \neq 0, a_8 = 0)$, $(a_7 = 0, a_8 \neq 0)$, $(a_7 = 0, a_8 = 0)$ by $\mathcal{X}$, $\mathcal{Y}$, and $\mathcal{Z}$ respectively. Then $\mathcal{K} = \mathcal{X} + \mathcal{Y} + \mathcal{Z}$. Our assumption is that $\dim(\mathcal{X}+\mathcal{Y}+\mathcal{Z}) \leq 8$ and we have already ruled out $\dim(\mathcal{X}+\mathcal{Y}+\mathcal{Z}) \leq 7$.

We show that $\dim (\mathcal{X} +  \mathcal{Y}) -\dim((\mathcal{X}+\mathcal{Y})\cap\mathcal{Z})) \leq 4$. If $\dim \mathcal{X} = 3$ then $M_2 =  
v_x^{\otimes 2} \in \mathcal{X}$, where $v_x = \begin{pmatrix} 1 & -\nicefrac{x_{13}^2}{x_{14}} & 0 & 0 \end{pmatrix}$.
This
$M_2$ has zeros outside of its top left $2 \times 2$ block, so it also lies in $\mathcal{Z}$. Hence $\dim (\mathcal{X} \cap \mathcal{Z}) \geq 1$. Similarly, if $\dim \mathcal{Y} = 3$ then $\dim (\mathcal{Y} \cap \mathcal{Z}) \geq 1$, since $v_y^{\otimes 2} \in \mathcal{Y} \cap \mathcal{Z}$, where $v_y = \begin{pmatrix} 1 & -\nicefrac{y_{13}^2}{y_{14}} & 0 & 0 \end{pmatrix}$.
Hence if $\dim \mathcal{X} = \dim \mathcal{Y} = 3$, then $\dim ( (\mathcal{X} + \mathcal{Y}) \cap \mathcal{Z}) \geq 2$, if  $v_x^{\otimes 2}$ and $v_y^{\otimes 2}$ are linearly independent.
If $v_x^{\otimes 2}$ and $v_y^{\otimes 2}$ are linearly dependent, then $\dim (\mathcal{X} + \mathcal{Y})\leq 5$ and $\dim ( (\mathcal{X} + \mathcal{Y}) \cap \mathcal{Z}) \geq 1$. Hence, in all cases, $\dim (\mathcal{X} + \mathcal{Y}) -\dim((\mathcal{X}+\mathcal{Y})\cap\mathcal{Z})) \leq 4$.

The previous paragraph, together with $\dim \mathcal{Z} \leq 3$, implies that $\dim \mathcal{K} \leq 7$, 
since $\dim(\mathcal{X}+\mathcal{Y}+\mathcal{Z}) \leq \dim (\mathcal{X} + \mathcal{Y}) -\dim((\mathcal{X}+\mathcal{Y})\cap\mathcal{Z})) + \dim \mathcal{Z}$. This is our required contradiction, hence $\drk_2 \Tcal \geq 9$.
\end{proof}

\subsection{Sylvester's rank inequality}

We use the decomposable flattening rank to lower bound the rank of a tensor, by combining it with Sylvester's rank inequality.

\begin{theorem}[Sylvester's rank inequality]
\label{thm:sylvester}
For matrices $A \in \RR^{m \times n}$ and $B \in \RR^{n \times k}$, 
$$\rk AB \geq \rk A + \rk B -n.$$
\end{theorem}

The inequality gives the first lower bound in Theorem~\ref{thm:A}, which we restate here.

\begin{theorem}\label{thm:general_syl}
Let $\Tcal$ be an order $d$ tensor, and fix $J \subset [d]$, with $J^c = [d] \backslash J$. Then
$$ \rk \Tcal \, \geq \, \drk_J \Tcal + \drk_{J^c} \Tcal - \rk \Tcal^{(J)} . $$
\end{theorem}

\begin{proof}
Set $r:= \rk\Tcal$ and fix a decomposition $\Tcal=\sum_{i=1}^r x_i^{(1)}\otimes\cdots \otimes x_i^{(d)}$.  Without loss of generality $J=\{1,\ldots,j\}$. 
Then,
\begin{align}
~\label{eqn:TJdecomp}
\Tcal^{(J)}&=\sum_{i=1}^r \Vect(x_i^{(1)}\otimes\cdots\otimes x_i^{(j)}) \otimes 
\Vect(x_i^{(j+1)}\otimes\cdots\otimes x_i^{(d)})= U S ,
\end{align}
where $U$ is the $(I_1 \cdots I_j) \times r$ matrix with $i$th column $\Vect (x_i^{(1)}\otimes\cdots\otimes x_i^{(j)})$, and $S$ is the $r \times (I_{j+1} \cdots I_d)$ matrix with $i$th row 
$\Vect(x_i^{(j+1)}\otimes\cdots\otimes x_i^{(d)})$.
Applying Sylvester's rank inequality to~\eqref{eqn:TJdecomp} gives
$\rk\Tcal^{(J)} \geq \rk U + \rk S -r$.

Every column of $\Tcal^{(J)}$ is a linear combination of the columns of $U$, which are decomposable. Choose a subset of the columns of $U$ that are linearly independent. This gives an expression for each column of $\Tcal^{(J)}$ as a linear combination of $\rk U$ (vectorised) decomposable tensors. Hence $\rk U \geq \drk_J \Tcal$. 
Similarly, each row of $\Tcal^{(J)}$ (i.e. each column of $\Tcal^{(J^c)}$) is a linear combination of the rows of $S$. The rows of $S$ are (vectorised) decomposable tensors, hence $\rk S \geq \drk_{J^c} \Tcal$.
In conclusion, $\rk\Tcal^{(J)}\geq\drk_J \Tcal + \drk_{J^c} \Tcal -r$.
\end{proof}

\begin{remark}
Theorem~\ref{thm:general_syl} gives inequalities among the ranks of certain unfoldings of a tensor. 
Given  $\Tcal \in \RR^{I_1} \otimes \cdots \otimes \RR^{I_d}$, the unfoldings of $\Tcal$ are indexed by partitions of $[d]$, see~\cite{wang2017operator} and the discussion after Definition~\ref{def:flatten}. 
Let $J = \{ 1, \ldots, j\}$, for ease of notation. 
Then Theorem~\ref{thm:general_syl} compares the flattening indexed by partition
$\{ 1, \ldots, j \} \cup \{ j + 1, \ldots, d\}$ with the unfoldings of $\{ 1 \} \cup \ldots \cup \{ j \} \cup \{ j+1, \ldots, d \}$ and $\{ 1, \ldots, j \} \cup \{ j + 1 \} \cup \ldots \cup \{ d \}$.
\end{remark}

\begin{remark}
There are other applications of Sylvester's rank inequality in the study of tensor rank.
It is used to show that the rank of a generic tensor is equal to the rank of its $(I_1 \cdots I_j) \times (I_{j+1} \cdots I_d)$ flattening, provided $\rk \Tcal \leq \min (I_1 \cdots I_j , I_{j+1} \cdots I_d)$, in~\cite[Equation (17)]{calvi2019tight}.
It is used in the study of CUR decomposition~\cite{mahoney2009cur} of tensors in~\cite{cai2021mode}. It is used in multilinear rank decompositions in~\cite{domanov2020uniqueness} and in the context of orthogonal tensor decomposition in~\cite{anandkumar2013tensor}.
\end{remark}

We return to the polynomial~\eqref{eqn:quartic_ex}. Later, we will see that the following lower bound holds with equality.

\begin{corollary}
\label{cor:geq12}
The tensor $\Tcal = x^4-3y^4+ 12x^2yz+12xy^2w$ has $\rk \Tcal \geq 12$.
\end{corollary}

\begin{proof}
Theorem~\ref{thm:general_syl} gives 
$\rk \Tcal \geq 2 \drk_2 \Tcal - \rk \Tcal^{(2)}$. 
We have $\drk_2 \Tcal = 9$, 
by Proposition~\ref{prop:drk9}.
The slice space $\Lcal_2$ in~\eqref{eqn:4x4} is six-dimensional, i.e. the flattening $\Tcal^{(2)} \in \RR^{16 \times 16}$ has rank~$6$. Hence $r \geq 18 - 6 = 12$. 
\end{proof}

\subsection{Symmetric decomposable rank}

Recall the symmetric decomposable flattening rank from Definition~\ref{def:sdrk}.

\begin{proposition}\label{prop:sdrkineq}
For $1 \leq j \leq d$, we have $\rk \Tcal^{(j)} \leq \drk_j \Tcal \leq \sdrk_{j} \Tcal \leq \srk \Tcal$.
\end{proposition}

\begin{proof}
The inequality $\drk_j \Tcal \leq \sdrk_j \Tcal$ follows from the definitions and 
$\rk \Tcal^{(j)} \leq \drk_j \Tcal$ is in Proposition~\ref{prop:drkineq}. Let $r:= \srk\Tcal$ with $\left\{x_i^{\otimes d} \mid i\in\{1,\ldots,r\} \right\}$ tensors in a symmetric decomposition of $\Tcal$. Then, $\left\{x_i^{\otimes j} \mid i\in\{1,\ldots,r\} \right\}$ spans the slice space $\Lcal_j$, and is a set of symmetric decomposable tensors, so $\sdrk_{j} \Tcal \leq \srk \Tcal.$ 
\end{proof}

The $(j,1)$ partially symmetric rank of an order $j+1$ tensor is the smallest $r$ such that the tensor can be written as a linear combination of decomposable tensors of the form $x^{\otimes j} \otimes y$, see~\cite{gesmundo2019partially}.
The following is proved for $j = d-1$, and for ranks defined over the complex numbers, in~\cite[Corollary 2.5]{gesmundo2019partially}. 

\begin{proposition}
The symmetric decomposable flattening rank $\sdrk_j \Tcal$ is the $(j,1)$ partially symmetric rank of the order $j + 1$ tensor whose $j+1$ slices are the order $j$ slices of $\Tcal$.
\end{proposition}

\begin{proof}
Let $\{\Ucal_1,\ldots,\Ucal_{r}\}$ be symmetric decomposable tensors whose linear span contains~$\Lcal_j$, where $r=\sdrk_j \Tcal$. 
Let $\Scal \in (\RR^{I})^{\otimes j} \otimes \RR^{I^{d-j}}$ be the order $j+1$ tensor from the statement.
Each $j+1$ slice of $\Scal$ is a linear combination of the $\Ucal_i$. This gives an expression for $\mathcal{S}$ as the sum of $r$ terms, as in~\eqref{eqn:Stensor}, with the required symmetry.
Conversely, if $\Scal$ has partially symmetric rank $r'$, then $\Scal$ is a linear combination of decomposable tensors $\{x_i^{\otimes j}\otimes y_i \mid i\in\{1,\ldots,r' \}\}$. Each $j+1$ slice is spanned by $\{x_i^{\otimes j} \mid i\in\{1,\ldots,r' \}\}$, which means $r'\geq \sdrk_j \Tcal$. Hence, $\sdrk_j \Tcal = \rk \Scal$.
\end{proof}

As in the non-symmetric case, we combine the symmetric decomposable flattening rank with Sylvester's rank inequality to lower bound the symmetric rank. 
This gives the second inequality from Theorem~\ref{thm:A}, which we restate here.

\begin{theorem}
\label{thm:symmetric_sylv}
Let $\Tcal$ be an order $d$ symmetric tensor, and fix $1 \leq j \leq d$. Then
$$ \srk \Tcal \geq \sdrk_j \Tcal + \sdrk_{d-j} \Tcal - \rk \Tcal^{(j)} . $$
\end{theorem}
\begin{proof}
Write $r:= \srk\Tcal$ and $\Tcal=\sum_{i=1}^r \la_i x_i^{\otimes d}$, where the $\la_i$ are non-zero scalars. Then
\begin{align*}
    \begin{split}
\Tcal^{(j)} & =
\begin{pmatrix}
\uparrow & &\uparrow\\
\Vect(x_1^{\otimes j}) &\cdots&\Vect(x_r^{\otimes j})
\\\downarrow & &\downarrow 
\end{pmatrix}
\begin{pmatrix}
\la_1&&\\
&\ddots&\\
&&\la_r
\end{pmatrix}
\begin{pmatrix}
\leftarrow & \Vect(x_1^{\otimes (d-j)})&\rightarrow
\\
 &\vdots& 
\\
\leftarrow & \Vect(x_r^{\otimes (d-j)})&\rightarrow
\end{pmatrix} \\
& = U \, \Lambda \, S.
\end{split}
\end{align*}
By Sylvester's rank inequality, $\rk\Tcal^{(j)}\geq\rk U +\rk (\Lambda S) -r=\rk U +\rk S -r.$ As in the proof of Theorem \ref{thm:general_syl}, we have $\rk U \geq \sdrk_j \Tcal$ and $\rk S \geq \sdrk_{d-j} \Tcal$.
\end{proof}

\subsection{Minimal rank and minimal symmetric rank}

Given a set of tensors $\mathcal{A}$, recall that $\minrk \mathcal{A}$ is the minimal rank of a tensor in $\mathcal{A}$.
Its symmetric analogue $\minsrk \mathcal{A}$ is the minimal symmetric rank of a symmetric tensor in $\mathcal{A}$.
In this section, we compare $\minrk \Acal$ and $\minsrk \Acal$.

\begin{proposition}
\label{prop:minrkA}
If $\minsrk \Acal \leq 1$ then
$\minsrk \Acal = \minrk \Acal$. 
\end{proposition}

\begin{proof}
If $\minrk \Acal = 0$, the zero tensor lies in $\Acal$. Since the zero tensor is symmetric, this implies $\minsrk \Acal = 0$. Hence $\minsrk \Acal > 0$ implies $\minrk \Acal > 0$. 
The inequality $\minrk \Acal \leq \minsrk \Acal$ then shows that $\minsrk \Acal = 1$ implies $\minrk \Acal = 1$. 
\end{proof}

We describe a linear space of tensors $\Ccal \Mod \Mcal$ with $\minrk (\Ccal \Mod \Mcal)$ strictly less than $\minsrk (\Ccal \Mod \Mcal)$. 
This example is extracted from~\cite[Section 5]{shitov2020comon}.

\begin{proposition}
\label{prop:shitov_idea} 
Let $\Ccal := x^4 - 3 y^4$ and $\Mcal := \{ x^2y,xy^2 \}$. Then $\minrk ( \Ccal \Mod \Mcal) < \minsrk (\Ccal \Mod \Mcal)$. \end{proposition}

\begin{proof}
We show that the linear space of tensors $\Ccal \Mod \Mcal$ contains a decomposable tensor but no symmetric decomposable tensor.
A symmetric tensor in $\Ccal \Mod \Mcal$
\begin{equation}
    \label{eqn:symUmodM}
     x^4 - 3y^4 + x^2 y (a x + b y) + xy^2 (c x + d y) , \qquad \text{for some} \quad a,b,c,d \in \RR.
\end{equation}
A symmetric decomposable $2 \times 2 \times 2 \times 2$ tensor with coefficient of $x^4$ equal to 1 can be written as
\begin{equation}
    \label{eqn:rankone}
    (x + \alpha y)^4 = x^4 + 4 \alpha x^3 y + 6 \alpha^2 x^2 y^2 + 4 \alpha^3 x y^3 + \alpha^4 y^4.
\end{equation} 
Equating the coefficient of $y^4$  in~\eqref{eqn:symUmodM} and~\eqref{eqn:rankone} gives 
$\alpha^4 = -3$, which has no real solutions. Hence $\minsrk (\Ccal \Mod \Mcal) \geq 2$.

We show that $\minrk (\Ccal \Mod \Mcal) \leq 1$. Adding $xy^2 + x^2y$ to the first $4$ slice and $-3(xy^2 + x^2 y)$ to the second $4$ slice of $\Ccal$ gives  the $2 \times 2 \times 2 \times 2$ tensor with $4$ slices
\begin{equation}
    \label{eqn:proof1}
    \left[
\begin{array}{cc||cc} 
1 & 1 & 1 & 1 \\ 
1 & 1 & 1 & 0 
\end{array}
\right] \qquad \text{and} \qquad 
\left[
\begin{array}{cc||cc} 
0 & -3 & -3 & -3 \\ 
-3 & -3 & -3 & -3 
\end{array}
\right] 
.
\end{equation}
Starting with a tensor of zeros, adding $x^2 y$ in multiples $a_1$, $a_2$, $a_3$, and $a_4$ to the first $1$ slice, $2$ slice, $3$ slice, and $4$ slice respectively gives the tensor with $4$ slices
\begin{equation}
    \label{eqn:proof2}
    \left[
\begin{array}{cc||cc} 
0 & \overline{a_2} & \overline{a_3} & 0 \\ 
\overline{a_1} & 0 & 0 & 0
\end{array}
\right] \qquad \text{and} 
\qquad \left[
\begin{array}{cc||cc} 
\overline{a_4} & 0 & 0 & 0 \\ 
0 & 0 & 0 & 0
\end{array}
\right] 
,
\end{equation}
where $\overline{a_i} := ( \sum_{j=1}^4 a_j) - a_i $.
Similarly, adding $x y^2$ in multiples $b_1$, $b_2$, $b_3$, and $b_4$ to the second $1$ slice, $2$ slice, $3$ slice, and $4$ slice respectively gives the tensor with $4$ slices
\begin{equation}\label{eqn:proof3}
\left[
\begin{array}{cc||cc} 
0 & 0 & 0 & 0 \\ 
0 & 0 & 0 & \overline{b_4}
\end{array}
\right] \qquad \text{and} 
\qquad \left[
\begin{array}{cc||cc} 
0 & 0 & 0 & \overline{b_1} \\ 
0 & \overline{b_3} & \overline{b_2} & 0
\end{array}
\right] 
,
\end{equation}
where $\overline{b_i} := ( \sum_{j=1}^4 b_j) - b_i $.
The sum of~\eqref{eqn:proof1},~\eqref{eqn:proof2}, and~\eqref{eqn:proof3} is decomposable when $(a_1, a_2 ,a_3, a_4) = (-1,-1,-1,2)$ and $(b_1, b_2, b_3, b_4) = (\frac13, \frac13, \frac13, -\frac23)$.
\end{proof}

\begin{remark}
\label{rem:CmodM}
Proposition~\ref{prop:shitov_idea} generalises to $\Ccal = x^d - 3 y^d$, $\Mcal=\{ x^{d-2}y, \ldots, xy^{d-2} \}$ for any even $d 
\geq 4$, as follows.
The comparison of~\eqref{eqn:symUmodM} and~\eqref{eqn:rankone} generalises to give $\minsrk (\Ccal \Mod \Mcal) = 2$. Moreover, $\minrk (\Ccal \Mod \Mcal) = 1$, see~\cite[Lemma 5.11]{shitov2020comon}.
These results also hold for $\Ccal = x^d - k y^d$, for any $k \in \RR_{>0}$.
\end{remark}

\subsection{The symmetric substitution conjecture}

We use the minimal rank and minimal symmetric rank to study tensors $\SAdj(\Ccal,\Mcal)$, see Definition~\ref{def:adj}.
We have
$$ \rk \SAdj(\Ccal,\Mcal) \geq \minrk (\Ccal \Mod \Mcal) + d \dim \Span \Mcal,$$ by Corollary~\ref{cor:adoined_substitution}. We conjecture its symmetric analogue, the real analogue to~\cite[Conjecture 7]{shitov2018counterexample}.

\begin{conjecture}[{The real symmetric substitution conjecture}]
\label{conj:modspan}
Fix a symmetric tensor $\Ccal \in (\RR^{I})^{\otimes d}$ and a finite set of symmetric tensors $\Mcal\subset (\RR^{I})^{\otimes (d-1)}$. Then
\begin{equation}
    \label{eqn:conj_ineq}
    \srk \SAdj(\Ccal,\Mcal) \geq \minsrk (\Ccal \Mod \Mcal) + d \dim  \Span \Mcal  .
\end{equation} 
Equality holds if $\Mcal$ consists of decomposable tensors.
\end{conjecture}

\begin{proposition}
\label{prop:conj_lower}
Fix a symmetric tensor $\Ccal \in (\RR^{I})^{\otimes d}$, with 
$\Mcal \subset (\RR^{I})^{\otimes (d-1)}$ a finite set of symmetric decomposable tensors. Then
\[ \srk \SAdj ( \Ccal, \Mcal) \leq \minsrk (\Ccal \Mod \Mcal) + d \dim  \Span \Mcal. \] 
\end{proposition}

\begin{proof}
Let $k := \dim \Span \Mcal$. 
Reorder so that the first $k$ tensors in $\Mcal$ are linearly independent and
denote the $i$th tensor in $\Mcal$ by $v_i^{\otimes (d-1)}$.
Let~$\Tcal \in (\RR^I)^{\otimes d}$ be a tensor of minimal symmetric rank in $\Ccal\Mod\Mcal$. We view $\Tcal$ as a tensor in $(\RR^{I \cup W})^{\otimes d}$ under the inclusion of index sets $I \subset I \cup W$, this is called padding in~\cite[Definition 7.6]{shitov2020comon}.
Then  
\begin{equation}
\label{eqn:dk_decomp}
\SAdj(\Ccal,\Mcal) =  \Tcal +
    \sum_{i = 1}^k \left( v_i^{\otimes (d-1)} \otimes w_i^{(d)} + \cdots + w_i^{(1)} \otimes v_i^{\otimes (d-1)} \right),
\end{equation} 
for some $w_i^{(j)} \in \RR^{I\cup W}$, where $i \in \{ 1, \ldots, k\}$ and $j \in \{ 1, \ldots, d\}$.
Permuting indices in~\eqref{eqn:dk_decomp} gives another expression for the symmetric tensor $\SAdj(\Ccal,\Mcal)$. 
Averaging over all rotations of indices, gives
\begin{equation}
    \label{eqn:dk_decomp_sym}
   \SAdj(\Ccal,\Mcal) =  \Tcal +
    \sum_{i = 1}^k \ell_{v_i}^{d-1} \ell_{w_i},
\end{equation}
where the coefficients of $\ell_{v_i}$ and $\ell_{w_i}$ are the vectors $v_i$ and $w_i = \frac{1}{d} ( w_i^{(1)} + \cdots + w_i^{(d)})$.
Each tensor $\ell_{v_i}^{d-1} \ell_{w_i}$ has symmetric rank $d$, since $v_i \neq w_i$. The symmetric rank of $\SAdj(\Ccal,\Mcal)$ is therefore at most~$\rk \Tcal + dk$.
\end{proof}

\begin{proposition}
If $\minsrk (\Ccal \Mod \Mcal) \leq 1$ then Conjecture~\ref{conj:modspan} holds.
\end{proposition}

\begin{proof}
Corollary~\ref{cor:adoined_substitution} gives $\rk\SAdj(\Ccal,\Mcal) \geq \minrk (\Ccal \Mod \Mcal) + d\dim\Span \Mcal$. This is the lower bound in the conjecture, since $\minrk (\Ccal \Mod \Mcal) = \minsrk (\Ccal \Mod \Mcal)$ by Proposition~\ref{prop:minrkA}.
Equality when $\Mcal$ consists of decomposable tensors is Proposition~\ref{prop:conj_lower}.
\end{proof}

When the tensors in $\Mcal$ are decomposable,~\eqref{eqn:dk_decomp_sym} is an expression for $\SAdj(\Ccal,\Mcal)$, where $\Tcal$ is a tensor of minimal symmetric rank in $\Ccal \Mod \Mcal$. 
Since 
the linear powers $\{ \ell_{v_i}^{d-1} \mid i \in \{ 1, \ldots, k \} \}$
are a basis of $\Mcal$, 
they
 are linearly independent. The linear forms $\{ \ell_{w_i} \mid i \in \{ 1, \ldots, k \} \}$ are also linearly independent, since their coordinates in $W$ give the coefficient of $v_i^{\otimes (d-1)}$ in each element of $\Mcal$. In the presence of further linear independence assumptions, we can prove Conjecture~\ref{conj:modspan}.

\begin{proposition}
Fix $\SAdj(\Ccal,\Mcal) = \Tcal +
    \sum_{i = 1}^k \ell_{v_i}^{d-1} \ell_{w_i}$,
where $\Tcal = \sum_{j=1}^r x_j^{\otimes d}$ is a tensor of minimal symmetric rank in $\Ccal \Mod \Mcal$.
If the linear forms $\ell_{v_i}, \ell_{w_i}, x_j$ are all linearly independent, for $i \in \{1, \ldots, k\}$ and $j \in \{ 1, \ldots, r\}$, then Conjecture~\ref{conj:modspan} holds.
\end{proposition}

\begin{proof}
As in the proof of Proposition~\ref{prop:conj_lower}, we view $\Tcal \in (\RR^I)^{\otimes d}$ as a tensor in $(\RR^{I\cup W})^{\otimes d}$.
Complex rank lower bounds real rank.
 The complex symmetric rank
of $\sum_{i=1}^r x_i^{\otimes d} + 
\sum_{i = 1}^k \ell_{v_i}^{d-1} \ell_{w_i}$ is $r+dk$, by \cite[Theorem 3.2]{Carlini2012TS}, since it is a sum of coprime monomials, $r$ of rank one and $k$ of rank $d$.
\end{proof}

\subsection{Comparison of lower bounds}

Theorem~\ref{thm:A} gives lower bounds on the rank and symmetric rank of a tensor, by combining the decomposable flattening rank with Sylvester's rank inequality. In this section, we compare these lower bounds to those of the substitution method (Theorem~\ref{thm:substitution} and Conjecture~\ref{conj:modspan}). We see that Theorem~\ref{thm:A} can prove Conjecture~\ref{conj:modspan} in special cases. We also compare to the lower bounds from a single unfolding and to~\cite{landsberg2010ranks}.

 \begin{lemma}\label{lem:decomp_mu}
Fix $f = x^{d-1}(\alpha x + d y)$.
The rank $d$ symmetric decompositions of $f$ are
$$
\sum_{i=1}^d\frac{(\la_ix+y)^d}{\prod_{j:j\neq i}(\la_i-\la_j)} , \quad \text{ where } \la_1,...,\la_d \in \RR \text{ are distinct and } \alpha = \sum_{i=1}^d \lambda_i.$$
\end{lemma}

\begin{proof}
The polynomial $x^{d-1} y$ has rank $d$~\cite[Proposition 5.6]{comon2008symmetric}. Hence $f$ has rank~$d$ for all $\alpha$, since the rank is unchanged by invertible change of basis. This means there does not exist a rank $d$ decomposition of $f$ with summand $\lambda x^d$: if there were, we would have a symmetric decomposition of 
$x^{d-1}( (\alpha-\lambda) x + dy)$
of rank $d-1$.
Hence we restrict to decompositions
$
\sum_{i=1}^d\mu_i(\la_ix+y)^d,
$
for scalars $\mu_i$ and $\lambda_i$. Equating coefficients, finding a decomposition is equivalent to finding a linear relation, with non-zero coefficient of the first row, among the the rows of the $(d + 1) \times (d+1)$ matrix
$$
A=\begin{pmatrix}
\alpha &1&0&\cdots&0\\
\la_1^d&\la_1^{d-1}&\la_1^{d-2}&\cdots&1\\
\vdots& & & &\vdots\\
\la_d^d&\la_d^{d-1}&\la_d^{d-2}&\cdots&1
\end{pmatrix}. \qquad \text{Let} \quad
 B = \begin{pmatrix}
\la_1^d&\la_1^{d-2}&\ldots&1\\
\vdots& & &\vdots\\
\la_d^d&\la_d^{d-2}&\ldots&1
\end{pmatrix},
$$
then $\det A = \alpha \det V - \det B$, where $V$ is the $d \times d$ Vandermonde matrix. 
The ratio $\frac{\det B}{\det V}$ is $(\lambda_1 + \cdots + \lambda_d)$, as follows. Both $\det B$ and $\det V$ are alternating functions, with $\det B$ degree one higher than $\det V$. Hence their ratio is a symmetric function of degree~$1$, a scalar multiple of $(\lambda_1 + \cdots + \lambda_d)$. It remains to compare coefficients to see that the scalar multiple is one.
Hence $\det A =(\alpha-(\la_1+\ldots+\la_d))\det V$, cf.~\cite[Proposition 5.6]{comon2008symmetric}.

The condition $\alpha=\la_1+\cdots+\la_d$ holds on the component of the solution that uses a non-zero multiple of the first row. To find $\mu_1,\ldots,\mu_d$, we write
$$
\begin{pmatrix}
\mu_1&\mu_2&\ldots&\mu_d
\end{pmatrix}
\begin{pmatrix}
\la_1^{d-1}&\la_1^{d-2}&\ldots&1\\
\vdots& & & \vdots \\
\la_d^{d-1}&\la_d^{d-2}&\ldots&1
\end{pmatrix}=
\begin{pmatrix}
1&0&\ldots&0
\end{pmatrix}
$$
By Cramer's rule, we conclude that $\mu_i=\frac{(-1)^{i+1}\det A_{i1}}{\det A}=(\prod_{j:j\neq i}(\la_j-\la_i))^{-1}$ where $A_{ij}$ is the sub-matrix of $A$ with $i$th row and $j$th column deleted.
\end{proof}

\begin{proposition}
Assume $d = 2 \delta$ is even, let $\Mcal = \{ v^{\otimes (d-1)} \}$, and let $\Tcal = \sum_{j=1}^r x_j^{\otimes d}$ be a tensor of minimal symmetric rank in $\Ccal \Mod \Mcal$.
If $x_1^{\otimes \delta},\ldots,
x_r^{\otimes \delta}, v^{\otimes \delta}$ are linearly independent, then Conjecture~\ref{conj:modspan} holds for $\SAdj(\Ccal, \Mcal)$.
\end{proposition}

\begin{proof}
Let $\Ucal = \SAdj(\Ccal,\Mcal)$.
Conjecture~\ref{conj:modspan} is the inequality $\srk \Ucal \geq r + d$, since $\dim \Span \Mcal = 1$.
We write $\Ucal = \sum_{j=1}^r x_j^{\otimes d} + v^{d-1} w$, where $v^{d-1}w$ is shorthand for
 $v^{\otimes (d-1)} \otimes w + v^{\otimes (d-2)} \otimes w \otimes v + \cdots + w \otimes v^{\otimes (d-1)}$.
The slice space of order $\delta$ slices of $\Ucal$ is
\begin{equation}
    \label{eqn:Lj_ex}
    \Lcal_\delta = \langle \, x_1^{\otimes \delta}, \ldots, x_r^{\otimes \delta}, \, v^{\otimes \delta}, \, v^{\delta-1} w \rangle.
\end{equation}
The vector
$w$ is not in $\langle x_1,\ldots,x_r, v \rangle$, since it has a non-zero component along the adjoined basis vector. 
Hence $\Lcal_\delta$ is a linear space of dimension $r + 2$, i.e. $\rk \Ucal^{(\delta)} = r + 2$.
We therefore have the inequality $\srk \Ucal \geq 2\sdrk_\delta \Ucal -(r+2)$, by Theorem~\ref{thm:A}. 
It remains to show that $\sdrk_\delta \Ucal \geq r+\delta + 1$. At least $r+1$ rank one tensors are needed to span the subspace $\langle x_1^{\otimes \delta}, \ldots, x_r^{\otimes \delta}, v^{\otimes \delta} \rangle$, since all the rank one tensors appearing in it are linearly independent, by assumption. It remains to consider $v^{\delta-1} w$. 

A decomposition of $v^{\delta-1}w$ must have at least $\delta$ linearly independent rank one terms, by Lemma~\ref{lem:decomp_mu}. Project the decomposition to the subspace $\langle v, w \rangle$ and consider it in the basis $\{ v, w \}$. In at least $\delta$ terms in the decomposition, the vector $w$ has non-zero coefficient, 
by the proof of Lemma~\ref{lem:decomp_mu}. Each of these $\delta$ terms are not in the span of the others, hence $\sdrk_\delta \Ucal \geq r + 1 + \delta$.
\end{proof}

\begin{remark}
\label{rmk:conj_more}
We explain how Theorem~\ref{thm:A} might prove Conjecture~\ref{conj:modspan} for $k:= \dim \Span \Mcal > 1$. We need to show that at least $r + k (\delta + 1)$ decomposable symmetric tensors are needed to span $\Lcal_\delta$. The idea is to show that each new rank $\delta$ tensor ${v_i}^{\delta - 1} w_i$ from~\eqref{eqn:dk_decomp_sym} requires at least $\delta$ new decomposable tensors. The challenge is to rule out the possibility of overlap between the different decompositions.
\end{remark}
 
Both Theorem~\ref{thm:A} and the substitution method (Theorem~\ref{thm:substitution}) lower bound the rank of a tensor in terms of the rank of tensors of strictly smaller size or order. 
 In both approaches, there is a trade-off: larger, higher order tensors may give better lower bounds, but it is more difficult to find their rank.

We compare Theorem~\ref{thm:A} to the substitution method for the tensor 
$\Tcal = x^4-3y^4+ 12x^2yz+12xy^2w$ from~\eqref{eqn:quartic_ex}.
We see that
Theorem~\ref{thm:A} can give a better lower bound than the substitution method.
Corollary~\ref{cor:geq12} explains how Theorem~\ref{thm:A} gives a lower bound of $12$ on the rank of $\Tcal$. (Later, we will see that this bound holds with equality.) The lower bound is obtained via a study of a linear space of matrices, i.e. an order three tensor.  This is a better bound than can be obtained by using the substitution method to get an order three tensor from $\Tcal$ via the subtraction of slices.

\begin{proposition}
\label{prop:sub_compare} 
Using the substitution method to reduce $\Tcal = x^4-3y^4+ 12x^2yz+12xy^2w$ to an order three tensor gives, at best, the lower bound $\rk \Tcal \geq 11$.
\end{proposition}

\begin{proof}
In the substitution method, the order in which slices are subtracted does not impact the lower bound obtained. Hence we consider the minimum rank in a linear space of tensors spanned by  the $4$ slices of $\Tcal$. The slices are cubics proportional to
$$\Tcal_x = x^3+6xyz+3y^2z, \quad \Tcal_y = -y^3 + x^2 z + 2 xyw, \quad \Tcal_z = x^2 y, \quad \Tcal_w = xy^2.$$
In the linear space, the coefficient of one of the four slices must be $1$, see Theorem~\ref{thm:substitution}.
Hence the lower bound is at best $3+\max\{\rk\Tcal_x,\rk\Tcal_y,\rk\Tcal_z,\rk\Tcal_w\}$. 
We have $\srk xyz=4$ and $\srk y^2z =3$, so $\rk \Tcal_x \leq 3+4+1=8$. Similarly, $\rk \Tcal_y \leq8$. Moreover $\rk \Tcal_z =\rk \Tcal_w =~3$. Hence the lower bound we obtain is at best $3 + 8 = 11$.
\end{proof}

\begin{remark}
We consider other ways to lower bound $\rk \Tcal$ for the tensor $\Tcal$ in~\eqref{eqn:quartic_ex}. The highest rank unfolding corresponds to the partition $\{ 1 , 2 \} \cup \{ 3 \} \cup \{4 \}$. Its rank is $\drk_2 \Tcal$, which is $9$ by Proposition~\ref{prop:drk9}. The lower bound from~\cite[Theorem 1.3]{landsberg2010ranks} is, in the notation of~\cite{landsberg2010ranks}, at best $\phi_{2,2} + \dim\Sigma_s + 1 = 6 + 1 + 1 = 8$.
\end{remark}

\section{Constructing tensors whose rank and symmetric rank differ}
\label{sec:idea} 

A real counterexample to Comon's conjecture over the real numbers is a real tensor whose (real) rank and symmetric rank differ. The only previously known example is from~\cite{shitov2020comon}. In this section, we organise the results of~\cite{shitov2020comon} into three steps
\begin{enumerate}
    \item[Step 1.] Find $\Ccal \in (\RR^I)^{\otimes d}$ symmetric and $\Mcal \subset  (\RR^I)^{\otimes (d-1)}$
    a finite set of symmetric tensors with 
    \begin{equation}
        \label{eqn:strict_minrk}
        \minrk (\Ccal \Mod \Mcal) < \minsrk (\Ccal \Mod \Mcal).
    \end{equation}
    \item[Step 2.] Modify $\Ccal$ and $\Mcal$ so that~\eqref{eqn:strict_minrk} still holds and $\Mcal$ consists of decomposable tensors
\item[Step 3.] Prove Conjecture~\ref{conj:modspan} for $\SAdj(\Ccal, \Mcal)$.
\end{enumerate}
If these three steps hold, then $\Tcal := \SAdj(\Ccal,\Mcal)$ has $$ \rk \Tcal = \minrk (\Ccal \Mod \Mcal) + d k < \minsrk (\Ccal \Mod \Mcal) + d k = \srk \Tcal,$$ where $k = \dim \Span \Mcal$ and the first equality is from Corollary~\ref{cor:adoined_substitution}. We use the results of~\cite{shitov2020comon} to show that the three steps hold on a family of examples. We prove accompanying results to highlight the importance of the choices made in the construction.

\subsection{Step 1} 

We saw an example of a symmetric tensor $\Ccal \in (\RR^2)^{\otimes 4}$ and finite set of symmetric tensors $\Mcal \subset  (\RR^2)^{\otimes 3}$ with $\minrk (\Ccal \Mod \Mcal) < \minsrk (\Ccal \Mod \Mcal)$ in Proposition~\ref{prop:shitov_idea}, namely $\Ccal = x^4 - 3y^4$ and $\Mcal = \{ x^2y, xy^2\}$. 
For this $\Ccal$ and $\Mcal$,
\begin{equation}
\label{eqn:sadj_ex}
    \Tcal := \SAdj (\Ccal, \Mcal) = x^4-3y^4+ 12x^2yz+12xy^2w 
\end{equation}
is the polynomial from~\eqref{eqn:quartic_ex}.
Since $\minrk (\Ccal \Mod \Mcal)$ and $\minsrk (\Ccal \Mod \Mcal)$ differ, Corollary~\ref{cor:adoined_substitution} and Conjecture~\ref{conj:modspan} give different lower bounds on the rank and symmetric rank of $\Tcal$. Corollary~\ref{cor:adoined_substitution} gives $\rk \Tcal \geq 9$ and Conjecture~\ref{conj:modspan} gives $\srk \Tcal \geq 10$.
However, neither lower bound holds with equality and $\Tcal$ is not a tensor whose rank and symmetric rank differ.

\begin{proposition}
\label{prop:eq12}
Fix $\Tcal = x^4-3y^4+ 12x^2yz+12xy^2w$. Then $\rk \Tcal = \srk \Tcal = 12$.
\end{proposition}

\begin{proof}
Corollary~\ref{cor:geq12} showed $\rk \Tcal \geq 12$. Here we show that $\srk \Tcal \leq 12$, using the Apolarity Lemma, see e.g. \cite[Lemma 2.1]{carlini2017real}  or \cite[Lemma 1.15]{iarrobino1999power}.
We examine the structure of the apolar ideal of $\Tcal$ to impose structure on a possible rank 12 decomposition. This reduces the number of parameters in the decomposition, making it feasible to find a solution.

The two polynomials $f(x,y,z):= x^4+ 12x^2yz$ and $g(x,y,w):= -3y^4+ 12xy^2w$ have the same symmetric rank, since $g(y,x,-3z) = -3f(x,y,z)$. Since $\Tcal = f + g$, it suffices to show that $\srk f \leq 6$.
By the apolarity lemma, we seek vanishing ideals of points that are contained in 
the apolar ideal
$$ f^\perp =\langle x^5,y^2,z^2,x^3- xyz,x^3y,x^3z\rangle. $$
Since $y^2$ and $z^2$ are contained in $f^\perp$, we have
$ y^2 - a^2 z^2 = (y - az) (y + az) \in f^\perp$ for all constants $a$.
We restrict our attention to ideals of points that are contained in $y^2 - a^2 z^2$ for fixed~$a$. That is, we look for a decomposition $f = \sum_{i=1}^6 \lambda_i \ell_i^4$, where $\ell_i = b_i x \pm a y + z$. 
We equate coefficients of $f$ and the decomposition
\begin{equation}
\label{eqn:6decmp}
 f \quad = \quad \sum_{i=1}^3 \lambda_i (b_i x + a y + z)^4 \quad + \quad \sum_{i=4}^6 \lambda_i (b_i x - a y + z)^4
 \end{equation}
and set $(b_1,b_2,b_4,b_5) =(1,2,1,3)$. The system of equations can then be solved in mathematica or Macaulay2 to give $a = -3$ and the rank six decomposition
\begin{align}
\label{eqn:ranksix}
\begin{split}
f \quad = \quad & \frac{1}{24}\left(x-3y+z\right)^4-\frac{1}{30}\left(2x-3y+z\right)^4-\frac{1}{120}\left(-3x-3y+z\right)^4
\\
& -\frac{1}{60}\left(x+3y+z\right)^4+\frac{1}{84}\left(3x+3y+z\right)^4+\frac{1}{210}\left(-4x+3y+z\right)^4.
\end{split}
\end{align}
When looking for a general rank six decomposition, rather than one of the restricted form~\eqref{eqn:6decmp}, our computation did not terminate. 
 \end{proof}

\subsection{Step 2}
\label{aj_rankone}

We seek to modify $\Ccal$ and $\Mcal$ so that the lower bounds from Corollary ~\ref{cor:adoined_substitution} and Conjecture~\ref{conj:modspan} hold with equality. Equality holds (or is conjectured to hold) when the adjoined tensors are decomposable.
A first approach is therefore to replace $\Mcal$ by symmetric rank one tensors that span $\Mcal$. We show that such an approach breaks the strict inequality~\eqref{eqn:strict_minrk}.

\begin{proposition}
\label{prop:rank_one_slices}
Let $\Ccal = x^4 - 3 y^4$ and let $\Wcal$ be a finite set of symmetric decomposable tensors that spans $\Mcal = \{ x^2y, xy^2\}$. Then $\minrk \Ccal \Mod \Wcal = 0$.
\end{proposition}

\begin{proof}
To show that the zero tensor is in $\Ccal \Mod \Wcal$, it is enough to show that $\Wcal$ spans $\mathcal{K} = \{ x^3, x^2y, xy^2,  y^3 \}$, since then any slice of $\Ccal$ is in $\Span \Wcal$. If $\dim \Wcal =4$, then $\Wcal$ spans $\mathcal{K}$. It therefore suffices to rule out the possibility that $\dim \Wcal \leq 3$. 

Suppose for contradiction that $\dim \Wcal \leq 3$. 
Since $x^2y$ has rank $3$, 
we have $\dim \Wcal = 3$.
Then $(\la_1 x + \mu_1 y)^3, (\la_2 x + \mu_2 y)^3, (\la_3 x + \mu_3 y)^3$ are a basis for $\Wcal$. They must be the rank one terms in a decomposition for both $x^2y$ and  $xy^2$. By Lemma \ref{lem:decomp_mu}, $\la_1,\la_2,\la_3,\mu_1,\mu_2,\mu_3$ are non-zero and $\frac{\la_1}{\mu_1}+\frac{\la_2}{\mu_2}+\frac{\la_3}{\mu_3}=0, \frac{\mu_1}{\la_1}+\frac{\mu_2}{\la_2}+\frac{\mu_3}{\la_3}=0.$ Then $1=\frac{\mu_1}{\la_1}\frac{\la_1}{\mu_1}=(\frac{\la_2}{\mu_2}+\frac{\la_3}{\mu_3})(\frac{\mu_2}{\la_2}+\frac{\mu_3}{\la_3})=
2+ t + t^{-1}$, where $t =
\frac{\la_2 \mu_3}{\la_3 \mu_2}$. This function is either at least $4$ or at most $0$ so can never be $1$, the desired contradiction.
\end{proof}

The set $\Wcal$ from Proposition~\ref{prop:rank_one_slices} results in $\minrk (\Ccal \Mod \Wcal) = \minsrk (\Ccal \Mod \Wcal)$, cf. Proposition~\ref{prop:minrkA}. We need a different way to replace $\Mcal$ with decomposable tensors, in order to preserve the strict inequality in~\eqref{eqn:strict_minrk}.

\begin{definition}[{See \cite[Definition 6.3]{shitov2020comon} and \cite[Notation 1.1]{Shitov2019CounterexamplesTS}}]
Fix a binary tensor $\Tcal \in (\RR^2)^{\otimes d}$. 
Let  $E := \{1, \ldots, n\}$ and $\mathcal{E} := \{ n+1, \ldots, 2n\}$. 
The $n$ {\em clone} of $\Tcal$, denoted $\Tcal_c$, is the tensor in $(\RR^{2n})^{\otimes d} = (\RR^{E \cup \mathcal{E}})^{\otimes d}$ with entries
$$ \Tcal_c (k_1 | \cdots | k_d) = 
\Tcal( h_1 | \cdots | h_d), \quad
\text{where} \quad 
 h_i = 
\begin{cases}1 & k_i \in E \\
2 & k_i \in \mathcal{E}. \end{cases} $$
For $\Mcal \subset (\RR^2)^{\otimes d}$ we denote by $\Mcal_c \subset (\RR^{2n})^{\otimes d}$ the set of $n$ clones of each tensor in $\Mcal$.
\end{definition}

\begin{example}
The $2$ clone of the matrix
$$ \begin{pmatrix} 1 & 0 \\ 0 & 1 \end{pmatrix} \quad \text{ is } \quad
\begin{pmatrix} 1 & 1 & 0 & 0 \\ 1 & 1 & 0 & 0 \\ 0 & 0 & 1 & 1 \\ 0 & 0 & 1 & 1 \end{pmatrix}. $$
\end{example}

\begin{definition}[{See~\cite[Remark 6.6]{shitov2020comon}}]
Given $\Tcal \in (\RR^{E \cup \mathcal{E}})^{\otimes d}$, let $\Tcal_E \in (\RR^E)^{\otimes d}$ be its restriction to index set $E$. Similarly, for $\Wcal \subset (\RR^{E 
\cup \mathcal{E}})^{\otimes d}$, let $\Wcal_E$ denote the restriction of each tensor in $\Wcal$ to index set $E$. 
Denote the tensor in  $(\RR^{E})^{\otimes d}$ with all entries equal to $1$ 
by $\mathbb{I}(E,d)$.
For the set $\mathcal{E}$, define $\Tcal_\mathcal{E}$, $\Wcal_\mathcal{E}$, and $\mathbb{I}(\mathcal{E},d)$ similarly. 
\end{definition}

The following result gives conditions on the set of decomposable tensors $\Wcal$ such that the strict inequality~\eqref{eqn:strict_minrk} is preserved. It is extracted from~\cite{shitov2020comon}, in particular~\cite[Lemmas 6.5 and 8.14]{shitov2020comon}.
The numbering of conditions comes from~\cite[Definition 6.7]{shitov2020comon}.

\begin{proposition}
\label{prop:shitov1}
Fix $\Ccal = x^d - 3 y^d$ and $\Mcal=\{ x^{d-2}y, \ldots, xy^{d-2} \}$ for $d \geq 4$ even. 
Let $\Wcal \subset (\RR^{E \cup \mathcal{E}})^{\otimes (d-1)}$ be such that
\begin{itemize}
    \item[(3)] $\Span \Wcal$ contains the $n$ clone of every tensor in $\Mcal$
    \item[(4e)]  $\mathbb{I}(E,d)$ is the only decomposable tensor in $\mathbb{I}(E,d) \Mod \Wcal_E$
    \item[(4$\epsilon$)] 
     $\mathbb{I}(\mathcal{E},d)$ is the only decomposable tensor in $\mathbb{I}(\mathcal{E},d) \Mod \Wcal_\mathcal{E}$.
\end{itemize}
Then $\minrk (\Ccal_c \Mod \Wcal) < \minsrk (\Ccal_c \Mod \Wcal)$.
\end{proposition}

\begin{proof}
We saw that $\minrk (\Ccal \Mod \Mcal) < \minsrk (\Ccal \Mod \Mcal)$ in Remark~\ref{rem:CmodM}.
Next we show that $\minrk (\Ccal_c \Mod \Wcal) = 1$, cf.~\cite[Proof of Lemma 6.1]{shitov2020comon}. 
By (3),
$$\minrk (\Ccal_c \Mod \Wcal) \leq \minrk (\Ccal_c \Mod \Mcal_c) . $$
Moreover, by the definition of cloning, $\minrk (\Ccal_c \Mod \Mcal_c)= \minrk (\Ccal \Mod \Mcal)=1$.
We can rule out $\minrk (\Ccal_c \Mod \Wcal) = 0$: this would imply that the zero tensor lies in $(\Ccal_c)_E \Mod \Wcal_E$, a contradiction to (4e), since $(\Ccal_c)_E = \mathbb{I}(E,d)$. 

Finally, we show that $\minsrk (\Ccal_c \Mod \Wcal) = 2$. Assume for contradiction that there is a decomposable symmetric $\Tcal$ in $\Ccal_c \Mod \Wcal$. Then
 $\Tcal_E \in (\Ccal_c)_E \Mod \Wcal_E$, and therefore $\Tcal_E \in \mathbb{I}(E,d) \Mod \Wcal_E$. 
 Hence $\Tcal_E = \mathbb{I}(E,d)$, by (4e).
 Similarly, 
 $\Tcal_\mathcal{E} \in (\Ccal_c)_\mathcal{E} \Mod \Wcal_\mathcal{E}$, i.e. $\Tcal_\mathcal{E} \in -3 \cdot \mathbb{I}(\mathcal{E},d) \Mod \Wcal_\mathcal{E}$. Hence $\Tcal_\mathcal{E} = -3 \cdot \mathbb{I}(\mathcal{E},d)$, by (4$\epsilon$).
Since both diagonal blocks of $\Tcal$ are clones, and $\Tcal$ is decomposable, the tensor $\Tcal$ must be a clone, i.e. $\Tcal = \Ucal_c$ for some decomposable $\Ucal \in (\RR^2)^{\otimes d}$, see~\cite[Lemma 6.5]{shitov2020comon}. The tensor $\Ucal$ has $\Ucal(1 | \cdots | 1) = 1$ and $\Ucal (2 | \cdots | 2) = -3$, hence $\Ucal$ is not decomposable, the desired contradiction. 
 \end{proof}

\subsection{Step 3}
We have seen conditions on a set $\Wcal$ to preserve the strict inequality in~\eqref{eqn:strict_minrk}.
We aim to use this strict inequality 
$\minrk (\Ccal_c \Mod \Wcal) < \minsrk (\Ccal_c \Mod \Wcal)$
to conclude a strict inequality between the rank and symmetric rank of $\SAdj(\Ccal_c, \Wcal)$. For this, we seek conditions for  Conjecture~\ref{conj:modspan}
to hold with equality.

\begin{proposition}
\label{prop:shitov2} 
Fix $\Ccal = x^d - 3 y^d$ and $\Mcal=\{ x^{d-2}y, \ldots, xy^{d-2} \}$ for $d \geq 4$ even.
Assume that $\Wcal$ is such that conditions (3),(4e), and (4$\epsilon$) from Proposition~\ref{prop:shitov1} hold. Moreover, assume that
\begin{itemize}
    \item[(2)] $\Wcal$ consists of decomposable tensors
    \item[(6)] Sets $u_{E}\otimes(\mathbb{R}^{E})^{\otimes (d-1)}$ and $\mathbb{I}(E,d)\Mod \mathcal{W}_E$ are disjoint for all 
    $u^{\otimes (d-1)} \in {\rm Span}\mathcal{W}$.
\end{itemize}
Then Conjecture~\ref{conj:modspan} holds for $\SAdj(\Ccal_c, \Wcal)$.
\end{proposition}

\begin{proof}
The upper bound $\srk \SAdj (\Ccal_c, \Wcal) \leq d \dim \Span \Wcal + 2$ is Proposition~\ref{prop:conj_lower}. We explain how the results of~\cite{shitov2020comon} give equality. Let $r = dk + 1$, where $k = \dim \Span \Wcal$ and assume for contradiction $\srk \SAdj(\Ccal_c,\Wcal) \leq r$. We transform the symmetric rank~$r$ decomposition into a decomposition of $r$ (possibly non-symmetric) rank one terms
\begin{equation}
    \label{eqn:shitov2_decomp}
    \SAdj(\Ccal_c, \Wcal) = \Tcal + \sum_{j=1}^d \sum_{w = 1}^{k} \Tcal_w^{(j)} ,
\end{equation} 
where $\Tcal \in (\RR^{E \cup \mathcal{E} \cup W})^{\otimes d}$ satisfies three conditions: (i) $\Tcal \in \Ccal_c \Mod \Wcal$ (in particular, $\Tcal$ is zero outside of the index set $E \cup \mathcal{E}$) (ii) $\Tcal$ is symmetric, and (iii) $\Tcal = \Ucal_c$ for some $\Ucal \in (\RR^2)^{\otimes d}$. 
Such a $\Tcal$ cannot be decomposable, by Proposition~\ref{prop:shitov1}, which contradicts $\srk \SAdj(\Ccal_c,\Wcal) \leq dk + 1$.

The procedure to build the new decomposition is~\cite[Procedure 8.6]{shitov2020comon}. The fact that Procedure 8.6 produces $\Tcal$ satisfying (i) and (iii) is the culmination of~\cite[Section 8]{shitov2020comon} in~\cite[Lemma 8.14]{shitov2020comon}. Part (ii) follows from~\cite[Claim 9.3 and Lemma 9.4]{shitov2020comon}.
It remains to show that~\cite[Claim 9.3]{shitov2020comon} works whenever
the conditions in our statement hold. 
In~\cite{shitov2020comon}, the author proves that Claim 9.3 holds for a {\em monomial emulator}~\cite[Definition 6.7]{shitov2020comon}, a finite set of tensors in $(\RR^{E \cup \mathcal{E}})^{\otimes d}$ that satisfies properties (2), (3), (4e), (4$\epsilon$), and (6), as well as
\begin{enumerate}
\item[(1)]
$\mathcal{W}$ is linearly independent,
\item[(5$e$)] $\mathbb{I}(E,d-1)$ is the only rank one tensor in $\mathbb{I}(E,d-1)+{\rm Span}\mathcal{W}_{E}$,
\item[(5$\epsilon$)] $\mathbb{I}(\mathcal{E},d-1)$ is the only rank one tensor in $\mathbb{I}(E,d-1)+{\rm Span}\mathcal{W_{E}}$,
\end{enumerate}
We show that $(5e)$ is implied by $(4e)$.
A decomposable $\Tcal \in \mathbb{I}(E,d-1)+{\rm Span}\mathcal{W}_E$ that is not equal to $\mathbb{I}(E,d-1)$ gives a decomposable tensor in $\mathbb{I}(E,d) \Mod \mathcal{W}_{E}$ that is not $\mathbb{I}(E,d)$ by setting each of the $|E|$ 1-slices equal to $\Tcal$. Similarly, $(4\epsilon)$ implies $(5\epsilon)$.

Finally, we can disregard property (1), as follows.
We can assume that $\Wcal$ consists of linearly independent tensors, by restricting to a linearly independent subset of $\Wcal$,  cf.~\cite[Observation 7.1]{shitov2020comon}. This does not affect the other properties (2)-(6).
\end{proof}

\begin{corollary}
\label{cor:counter_ex}
Fix $\Ccal = x^d - 3 y^d$ and $\Mcal=\{ x^{d-2}y, \ldots, xy^{d-2} \}$ for $d \geq 4$ even.
Let $\Wcal$ satisfy the conditions of Propositions~\ref{prop:shitov1} and~\ref{prop:shitov2}. Then $\SAdj (\Ccal_c, \Wcal)$ is a tensor whose rank and symmetric rank differ.
\end{corollary}

\begin{proof}
Corollary~\ref{cor:adoined_substitution} gives $\rk \SAdj (\Ccal_c, \Wcal) = d \dim \Span \Wcal + 1$, since
$\Wcal$ is a set of decomposable tensors. In comparison,
$\srk \SAdj (\Ccal_c, \Wcal) = d \dim \Span \Wcal + 2$, since 
$\minsrk (\Ccal_c \Mod \Wcal) = 2$ and
Conjecture~\ref{conj:modspan} holds for $\SAdj(\Ccal_c,\Wcal)$, by Propositions~\ref{prop:shitov1} and~\ref{prop:shitov2}. 
\end{proof}

\section{A counterexample of order 6}
\label{sec:order6} 

In this section we give an order $6$ counterexample to Comon's conjecture; i.e., 
we prove Theorem~\ref{thm:B}.
We define a set of symmetric order $5$ tensors that satisfies the conditions from Propositions~\ref{prop:shitov1} and~\ref{prop:shitov2}, namely:
\begin{itemize}
         \item[(2)] $\Wcal$ consists of decomposable tensors,
    \item[(3)] $\Span \Wcal$ contains the $n$ clone of every tensor in $\Mcal=\{ x^4 y, x^3 y^2, x^2 y^3, x y^4 \}$,
    \item[(4e)]  $\mathbb{I}(E,6)$ is the only decomposable tensor in $\mathbb{I}(E,6) \Mod \Wcal_E$,
    \item[(4$\epsilon$)] 
     $\mathbb{I}(\mathcal{E},6)$ is the only decomposable tensor in $\mathbb{I}(\mathcal{E},6) \Mod \Wcal_\mathcal{E}$,
    \item[(6)] Sets $u_{E}\otimes(\mathbb{R}^{E})^{\otimes 5}$ and $\mathbb{I}(E,6)\Mod \mathcal{W}_E$ are disjoint for all 
    $u^{\otimes 5} \in {\rm Span}\mathcal{W}$.
\end{itemize}

\begin{definition}[The set $\Wcal$]
\label{def:Wset}
For any $i\in \{1,...n\}$, let $\alpha_{i} \in \mathbb{R}^{2n}$ have $\alpha_i( 2i - 1) = \alpha_i (2i) = 1$, and all other entries zero.
Let $\Wcal_{1}$ be the set of tensors $u^{\otimes 5}$, where $u \in \RR^{4n}$ is one of
\begin{footnotesize}
$$ \begin{matrix}
(\al_{i_1}+\al_{i_{2}}+\al_{i_{3}}+\al_{i_{4}}|\,0)
 & 
(\al_{i_{1}}+\al_{i_{2}}+\al_{i_{3}}|\,0) & (\al_{i_{1}}+\al_{i_{2}}|\,0) & (\al_{i_1}|\,0) \\ 
(\al_{i_{1}}+\al_{i_{2}}+\al_{i_{3}}+\al_{i_{4}}|\,\al_{k_1})
 &    (\al_{i_{1}}+\al_{i_{2}}|\frac{(n-2)^2}{(n-3)(n-1)}\al_{k_1}) & (\al_{i_1}|\frac{n-2}{n-3}\al_{k_1}) & (\al_{i_1}|\,\frac{n-1}{n-4}\al_{k_1})\\
 
 (\al_{i_{1}}+\al_{i_{2}}+\al_{i_{3}}|\,\al_{k_1}+\al_{k_2}) & (\al_{i_{1}}+\al_{i_{2}}|\,\frac{n-2}{n-3}(\al_{k_1}+\al_{k_2})) &  (\al_{i_{1}}+\al_{i_{2}}|\,\frac{n-2}{n-4}\al_{k_1})
  & (0|\,\al_{k_1})  \\ 
 (\al_{i_{1}}+\al_{i_{2}}+\al_{i_{3}}|\,\frac{n-3}{n-4}\al_{k_1}) &  (\al_{i_{1}}+\al_{i_{2}}+\al_{i_{3}}|\frac{n-2}{n-1}\al_{k_1}) & 
 (\al_{i_{1}}|\,\frac{n-1}{n-3}(\al_{k_1}+\al_{k_2})) & (0|\al_{k_1}+\al_{k_2})  \\
 \end{matrix}$$ \end{footnotesize}
where $1\leq i_{1} < i_{2}< i_{3} < i_{4}\leq n$ and $1\leq k_1< k_2\leq n$.
Define the permutation
\begin{equation}
    \label{eqn:pi}
    \pi(i_1 | i_2 | \cdots | i_{2n}|k_1 | k_2 | \cdots | k_{2n})=(k_2| \cdots | k_{2n} | k_1 | i_2 | \cdots | i_{2n} | i_1).
\end{equation}
Let $\Wcal_2$ be the set of tensors of the form $u^{\otimes 5}$ where $u$ is the image of one of the above vectors under permutation $\pi$. Define $\Wcal:= \Wcal_{1} \cup \Wcal_{2}$. See Figure~\ref{fig:neq4} for an illustration.
\end{definition} 

\begin{figure}[hbtp]
    \centering
    \includegraphics[width=13cm]{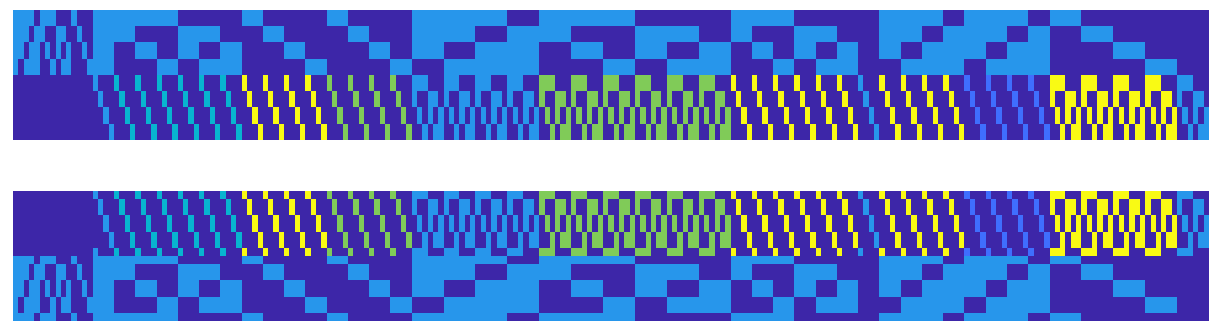}
    \caption{Let $n=4$. The sets $\Wcal_1$ and $\Wcal_2$ each consist of 225 tensors $u^{\otimes 5}$ for some $u \in \RR^{16}$. We illustrate these two $16 \times 225$ matrices of $u$ vectors as heatmaps.}  
    \label{fig:neq4}
\end{figure}

Definition~\ref{def:Wset} is the extension of~\cite[Lemma 11.4]{shitov2020comon} from order $4$ to $6$. We use it to find an order $6$ counterexample, which we now describe in more detail.

\begin{theorem}
\label{thm:order6counter}
 Let $\Ccal = x^6 - 3 y^6$ and let $\Wcal 
 \subset (\RR^{28})^{\otimes 5}$  be as in Definition~\ref{def:Wset}. Let $\Ccal_c$ be the $14$ clone of $\Ccal$. Then $\SAdj(\Ccal_c, \Wcal) \in (\RR^{5180})^{\otimes 6}$ has rank 30913 and symmetric rank 30914.
\end{theorem}

We prove Theorem~\ref{thm:order6counter}, and therefore Theorem~\ref{thm:B}, by showing that $\Wcal$ satisfies conditions (2), (3), (4$e$), (4$\epsilon$) and (6), provided $n \geq 7$. 
Condition (2) holds, since each tensor in $\Wcal$ is rank one. We show that the remaining conditions hold.

\begin{proposition}
\label{prop:cond3}
Condition (3) holds for $\Wcal$ in Definition~\ref{def:Wset}, provided $n \geq 5$.
\end{proposition}

\begin{proof} 
This proof is the $d=6$ analogue to~\cite[Lemma 11.7]{shitov2020comon}. The set $\Mcal$ is equal to $\{ x^4 y, x^3 y^2, x^2 y^3, x y^4 \}$.
The $n$ clones of $x^4 y$ and $x^3 y^2$ are in $\Span\Wcal_1$: 
see our matlab code \verb|github.com/seigal/loborrt|
for a numerical check, or Lemma \ref{lem:mu41} and Lemma~\ref{lem:mu32} for the algebraic identities.
Similarly, the clones of $x^2 y^3$ and $xy^4$ are in $\Span\Wcal_2$: for tensors that are clones, the permutation~$\pi$ in~\eqref{eqn:pi} just swaps the first $2n$ indices with the second $2n$ indices. Hence the $n$ clone of every tensor in $\Mcal$ is in $\Span \Wcal$.
\end{proof}

\begin{proposition}
\label{prop:4e}
Properties (4e) and (4$\epsilon$) hold for $\Wcal$ in Definition~\ref{def:Wset}, for $n \geq 7$.
\end{proposition}

\begin{proof}
This proof is the $d=6$ analogue to~\cite[Lemma 11.9]{shitov2020comon}. By symmetry, we only need to prove (4e). 
Every tensor in $\Wcal_E$ is zero at location $(k_1 | \cdots | k_{5})$, provided all ${{5 \choose 2}}$ differences $\delta_{ij} = (k_i - k_j) \Mod 2n$ satisfy $|\delta_{ij}| \geq 2$. Such entries exist provided $n \geq 5$. For example, all tensors in $\Wcal$ are zero at entry $(1|3|5|7|9)$. 

Let $\Tcal$ be a tensor in $\mathbb{I}(E,6)\Mod \Wcal_{E}$. Then $\Tcal(k_1| \cdots | k_6) =1$ whenever all the ${{6 \choose 2}}$ differences $\delta_{ij} = (k_i - k_j) \Mod 2n$ satisfy $|\delta_{ij}| \geq 2$. Such entries exist provided $n \geq 6$. For example, $\Tcal(1|3|5|7|9|11) = 1$, $\Tcal(2|4|6|8|10|12) = 1$, and $\Tcal(1|4|6|8|10|12) = 1$, and the entries of $\Tcal$ at all permutations of these indices are also $1$.

Assume that $\Tcal$ is decomposable, $\Tcal = u_1 \otimes \cdots \otimes u_6$.
Since $\Tcal(2|4|6|8|10|i) = 1$ for $i \in \{ 12,\ldots,2n \}$, we have  $u_6(12) = u_6(13) =. \cdots = u_6(2n)$.
This gives equality of multiple entries of $u_6$, provided $n \geq 7$. 
Similarly, $\Tcal(2n-1|2n-3|2n-5|2n-7|2n-9|i) = 1$ for $i \in \{ 1, \ldots, 2n-11\}$, hence we have $u_6(1) = u_6(2) = \cdots = u_6(2n-11)$. 
Other combinations of indices show that all 
entries of $u_6$ are equal. By a similar argument, all entries of the vectors $u_i$ are equal for $i \in \{ 1, \ldots, 5\}$. So all the entries of $\Tcal$ are equal. Since some entries of $\Tcal$ are one, we conclude that $\Tcal = \mathbb{I}(E,6)$.
\end{proof}

\begin{proposition}
\label{prop:cond6} 
Property (6) holds for $\Wcal$ in Definition~\ref{def:Wset}, provided $n \geq 6$.
\end{proposition}

\begin{proof}
We want to show that the sets $u_{E}\otimes(\mathbb{R}^{E})^{\otimes 5}$ and $\mathbb{I}(E,6)\Mod \mathcal{W}_E$ are disjoint for all $u^{\otimes 5} \in \Span \mathcal{W}$.
Fix $\Tcal = u^{\otimes 5} \in \Span \Wcal$. 
Then $\Tcal(1|3|5|7|9) = 0$, since this is true for every tensor in $\Wcal$, using the fact that $n \geq 5$.  Hence $u_E$ has some entry equal to zero, and 
so
$u_{E}\otimes(\mathbb{R}^{E})^{\otimes 5}$ contains a slice of zeros.
We show that every tensor in $\mathbb{I}(E,6)\Mod \mathcal{W}_E$ has a non-zero entry in every slice. Given an index $i$, consider the $(i|(i+2) \Mod 2n|\ldots|(i+12) \Mod 2n)$ entry of a tensor in  $\mathbb{I}(E,6)\Mod \mathcal{W}_E$. The difference between any pair of indices is at least $2$, since $n \geq 6$.
Hence, in any subset of 5 of these indices, every tensor in $\Wcal_E$ has a zero at that entry. Hence the $(i|(i+2) \Mod 2n|\ldots|(i+12) \Mod 2n)$ entry of any tensor in $\mathbb{I}(E,6)\Mod \mathcal{W}_E$ is $1$, cf.~\cite[Lemma 11.15]{shitov2020comon}.
\end{proof}

Next, we show that that tensors in $\Wcal$ are linearly independent. 
This is required to compute the rank and symmetric rank of the counterexample $\SAdj(\Ccal_c,\Wcal)$.
We also show that the tensors in the order $4$ example from~\cite{shitov2020comon} are linearly independent. This verifies the stated rank and symmetric rank for the order $4$ example from~\cite{shitov2020comon}.

\begin{lemma}
\label{lem:indep_li6}
Fix $\Tcal_1 \in \Span \Wcal_1$ and $\Tcal_2 \in \Span \Wcal_2$, where $\Wcal_1$ and $\Wcal_2$ are as in Definition~\ref{def:Wset}, with $n \geq 5$. If $\Tcal_1 + \Tcal_2 = 0$, then $\Tcal_1 = \Tcal_2 = 0$. 
\end{lemma}

\begin{proof}
A tensor in $\Span \Wcal_1$ has \begin{equation}
    \label{eqn:ii+1}
    \Tcal(i|k_2|\cdots|k_d) = \Tcal(i+1|k_2|\cdots|k_d),
\end{equation}
for $i \in \{1, 3, \ldots, 2n-1, 2n+1, 2n+3, \ldots, 4n-1\}$, and all $k_2, \ldots, k_d$, by the definition of the vectors $\alpha_i$ in Definition~\ref{def:Wset}. Similarly, a tensor in $\Wcal_2$ satisfies~\eqref{eqn:ii+1} for $i \in \{ 2, 4, \ldots, 2n-2 , 2n + 2, \ldots, 4n-2\}$ as well as $\Tcal(1|k_2|\cdots|k_d) = \Tcal(2n|k_2|\cdots|k_d)$ and $\Tcal(2n+1|k_2|\cdots|k_d) = \Tcal(4n|k_2|\cdots|k_d)$. 
The tensor $\Tcal_1$ lies in $\Span \Wcal_1$ and $\Span \Wcal_2$, since $\Tcal_1 = - \Tcal_2$. Then $\Tcal_1$ satisfies~\eqref{eqn:ii+1} for $i \in \{1, \ldots, 2n-1\} \cup \{2n+1, \ldots, 4n-1\}$. Moreover, $\Tcal_1$ is symmetric, so $\Tcal(\cdots|k_{j-1}|i|\cdots)=\Tcal(\cdots|k_{j-1}|i+1|\cdots)$ for $i \in \{1, \ldots, 2n-1\} \cup \{2n+1, \ldots, 4n-1\}$ for 
any $j \in \{2, \ldots , d\}$.
This is the condition for $\Tcal$ to be a clone: $\Tcal(k_1|\cdots|k_d)=\Tcal(k_1'|\cdots|k_d')$ if $k_i, k_i' \in \{1,\ldots , 2n\}$ or if $k_i, k_i' \in \{2n+1,\ldots, 4n\}$.
That is, $\Tcal_1 = \Ucal_c$ for some symmetric $\Ucal \in (\RR^2)^{\otimes 5}$.
The symmetric tensor $\Ucal$ is a binary quintic. 
We show that $\Ucal \in \langle x^4 y, x^3 y^2\rangle$.
We have $\Tcal_1(1|3|5|7|9) = \Tcal_1(2n+1|2n+3|2n+5|2n+7|2n+9) = 0$, provided $n \geq 5$, since $\Tcal_1 \in \Span \Wcal_1$. Hence the monomials $x^5$ and $y^5$ do not appear in $\Ucal$. Moreover, we have
$\Tcal_1(1|1|2n+1|2n+3|2n+5) = \Tcal_1(1|2n+1|2n+3|2n+5|2n+7) = 0$, since $\Tcal \in \Span \Wcal_1$. Hence the monomials $x^2 y^3$ and $x y^4$ do not appear in $\Ucal$. Therefore $\Ucal \in \langle x^4 y, x^3 y^2\rangle$. By a similar argument for $\Wcal_2$, we conclude $-\Ucal \in \langle x^2 y^3, x y^4\rangle$. Hence $\Ucal = 0$.
\end{proof}

\begin{proposition}
\label{prop:li_4}
The set of tensors defined in~\cite[Definition 11.4]{shitov2020comon} are linearly independent.
\end{proposition}

\begin{proof}
Denote the set by $\Wcal^{(4)} = \Wcal_1^{(4)} + \Wcal_2^{(4)}$. We show linear independence of $\Wcal^{(4)}_1$.
A linear combination of tensors in $\Wcal_1^{(4)}$ is 
\begin{align}
\begin{split}
    \label{eqn:linear_comb}
    \sum_{\substack{1\leq i<j\leq5\\ 1\leq k\leq5}}b_{ijk}(\al_i+\al_j|\al_k)^{\otimes3}+\sum_{1\leq i<j\leq 5}c_{ij}(\al_i+\al_j|0)^\otimes3
    +\sum_{\substack{1\leq i\leq5 \\ 1\leq k\leq5}}b_{ik}(3\al_i|4\al_k)^{\otimes3}
 \\ + \sum_{1\leq i\leq 5}c_i(\al_i|0)^{\otimes3}+\sum_{1\leq k \leq 5}b_k(0|\al_k)^{\otimes3} .
\end{split}
\end{align}
This is a $20 \times 20 \times 20$ tensor whose entries are linear combinations of the $95$ coefficients.
Setting~\eqref{eqn:linear_comb} to zero gives a system of $8000 = 20 \times 20 \times 20$ equations in $95$ unknowns. We show that the $95$ coefficients must all be zero in three steps, illustrated in Figure~\ref{fig:li_order4}.

In~\eqref{eqn:linear_comb}, $3840$ of the $8000$ tensor entries are zero. A further $2400$ entries are a single coefficient, the coefficients of the $50$ elements of $\Wcal^{(4)}_1$ of the form
$(\al_i+\al_j|\al_k)^{\otimes 3}$.
If~\eqref{eqn:linear_comb} is zero, these coefficients vanish.
Removing these terms from~\eqref{eqn:linear_comb} gives a linear combination of the remaining $45$ tensors in $\Wcal^{(4)}_1$. Repeating the argument, we have $1680$ entries of the tensor that are a single coefficient, the coefficients of $35$ tensors. Setting these to zero gives a linear combination of $10$ tensors in $\Wcal^{(4)}_1$, with $80$ non-zero entries, each equal to a single coefficient. These are the coefficients of the remaining $10$ vectors in $\Wcal^{(4)}_1$. Hence all $95 = 50 + 35 + 10$ tensors in $\Wcal^{(4)}_1$ have coefficient zero.

It remains to show that if $\Tcal_i \in \Wcal_i^{(4)}$ with $\Tcal_1 + \Tcal_2 = 0$, then $\Tcal_1 = \Tcal_2 = 0$. This is Lemma~\ref{lem:indep_li6} but in the order four case, with similar proof: a similar argument shows that $\Tcal_1 = - \Tcal_2$ is the clone of some $\Ucal \in (\RR^2)^{\otimes 3}$. Then $\Tcal_1 \in \Span \Wcal_1^{(4)}$ implies $\Ucal \in \langle x^2 y \rangle$ while $\Tcal_2 \in \Span \Wcal_2^{(4)}$ implies $-\Ucal \in \langle xy^2 \rangle$. Hence $\Ucal = 0$.
\end{proof}

\begin{figure}[hbtp]
    \centering
    \includegraphics[width=13cm]{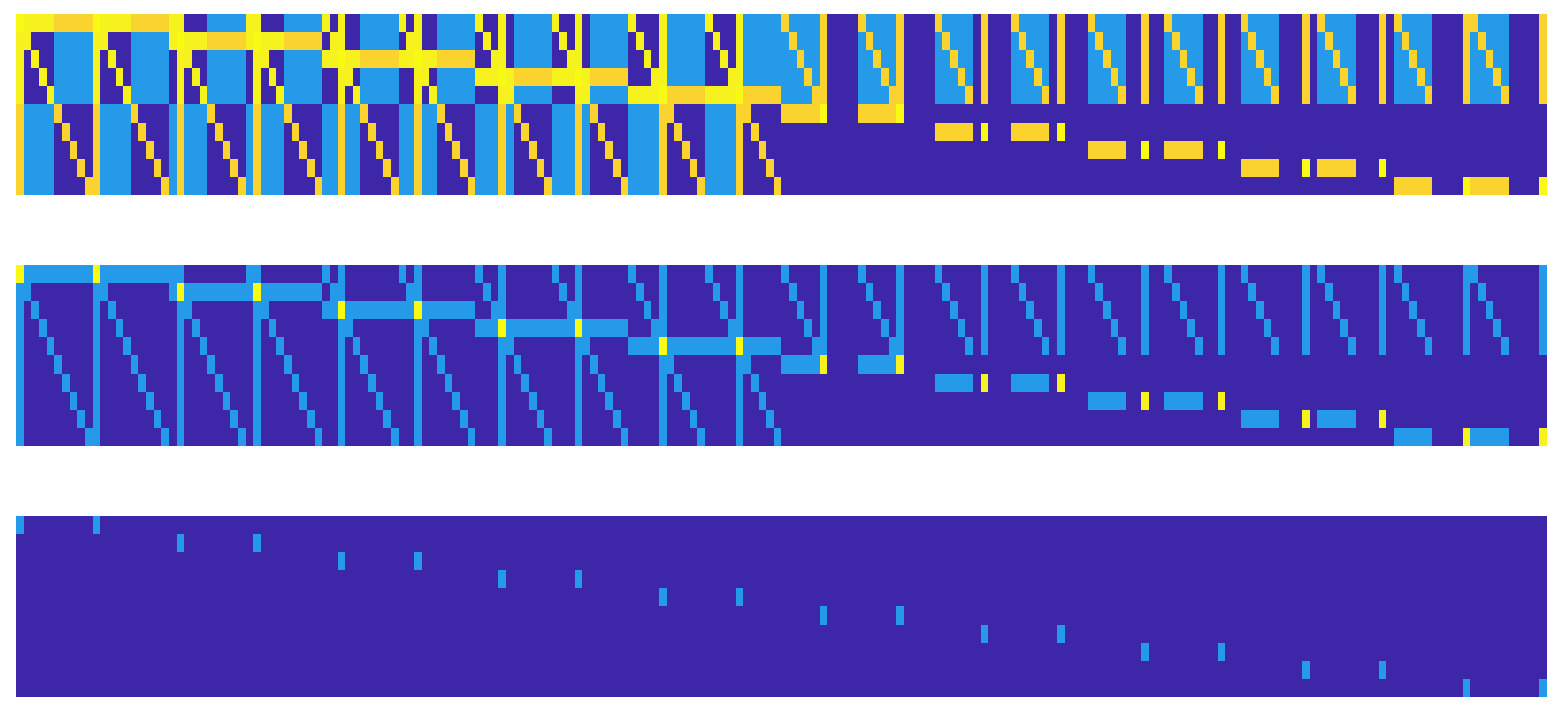}
    \caption{The proof of Proposition~\ref{prop:li_4} shows that the coefficients in~\eqref{eqn:linear_comb} are zero, in three steps. Each step studies a $20 \times 20 \times 20$ tensor of unknown coefficients, illustrated here as a $20 \times 400$ matrix. Darkest (dark blue) entries are zero, second darkest (light blue) entries are equal to one coefficient, and brightest (yellow/orange) entries are a linear combination of more than one coefficient.}
    \label{fig:li_order4}
\end{figure}

\begin{proposition}
\label{prop:li_6}
The set of tensors $\Wcal$ from Definition~\ref{def:Wset} is linearly independent.
\end{proposition}
\begin{proof}
We have $\Wcal = \Wcal_1 \cup \Wcal_2$. First we show that the vectors in $\Wcal_1$ are linearly independent. 
Consider a linear combination $\Tcal$ of vectors $u^{\otimes 6}$ where $u$ ranges over the $16$ types of vector in Definition~\ref{def:Wset}. Assume that this linear combination vanishes.

The only tensor in $\Wcal_1$ that is non-zero at entry  $(2i_1|2i_2|2i_3|2i_4|2n+2k)$ is $u^{\otimes 6}$ where $u = (\al_{i_1}+\al_{i_2}+\al_{i_3}+\al_{i_4}|\al_k)$. Hence no such terms appear in a vanishing linear combination. Having removed these terms, the only tensor in $\Wcal_1$ that is non-zero at entry 
$(2i_1|2i_2|2i_3|2i_4|2i_1)$
is $u^{\otimes 6}$ where $u = (\al_{i_1}+\al_{i_2}+\al_{i_3}+\al_{i_4}|0)$. Hence no such terms appear in a vanishing linear combination. 
The only tensor in $\Wcal_1$ with non-zero coefficient
$(2i_1|2i_2|2i_3|2k_1+2n|2k_2+2n)$ is $u^{\otimes 6}$, where $u = (\al_{i_1}+\al_{i_2}+\al_{i_3}|\al_{k_1} + \al_{k_2})$.  
 Hence no such terms appear in a vanishing linear combination. 
Repeating, by considering tensors in $\Wcal_1$ with smaller and smaller support, shows that all terms in the linear combination must have coefficient zero. 
By a similar argument, the set $\Wcal_2$ is linearly independent.
Now assume we have $\Tcal_1 \in \Wcal_1$ and $\Tcal_2 \in \Wcal_2$ with $\Tcal_1 + \Tcal_2 = 0$. Then $\Tcal_1 = \Tcal_2 = 0$, by Lemma~\ref{lem:indep_li6}.
\end{proof}

\begin{proof}[{Proof of Theorem~\ref{thm:order6counter}}.]
The tensor $\SAdj(\Ccal_c, \Wcal)$ has different rank and symmetric rank, by
Corollary~\ref{cor:counter_ex} and Propositions~\ref{prop:cond3},~\ref{prop:4e}, and~\ref{prop:cond6}.
It remains to find the size, rank, and symmetric rank of this tensor. The set of $\Wcal_1$ consists of $\binom{7}{4}+\binom{7}{3}+\binom{7}{2}+\binom{7}{1}+\binom{7}{4}\binom{7}{1}+\binom{7}{2}\binom{7}{1}+2\binom{7}{1}\binom{7}{1}+\binom{7}{3}\binom{7}{2}+\binom{7}{2}\binom{7}{2}+\binom{7}{2}\binom{7}{1}+\binom{7}{1}+2\binom{7}{3}\binom{7}{1}+\binom{7}{1}\binom{7}{2}+\binom{7}{2}=2576$ tensors.
Hence $\Wcal$ consists of $2576\times 2= 5152$ tensors. 
Therefore $\SAdj(\Ccal_c, \Wcal) \in (\RR^I)^{\otimes 6}$, where $|I| = 28 + 5152 = 5180$. 
The set $\Wcal$ is linearly independent, by Proposition~\ref{prop:li_6}. Hence $\rk \SAdj (\Ccal_c, \Wcal) = 1+5152\times6=30913$ and $\srk \SAdj (\Ccal_c, \Wcal) = 2+5152\times 6=30914$. 
\end{proof}

\begin{remark}
We can reduce the size of the tensor in Theorem~\ref{thm:order6counter} slightly, as follows.
Given $u=(u_E|u_{\Ecal}) \in \RR^{28}$ with $u^{\otimes 5} \in \Wcal$, the vectors $u_E, u_{\Ecal} \in \RR^{14}$ have the sum of their entries at even indices equal to the
sum of their entries at odd indices, hence the vectors $u = (u_E|u_{\Ecal})$ lie in a 26-dimensional subspace, cf.~\cite[Remark 11.2]{shitov2020comon}. 
So, with a change of basis, we have a counterexample in $(\RR^I)^{\otimes 6}$, where $|I| = 28 - 2 + 5152 = 5178$. 
\end{remark}

\begin{remark}
The border rank of the tensor $\SAdj(\Ccal_c, \Wcal) \in (\RR^{5180})^{\otimes 6}$ is at most $2+5152\times2 = 10306$, since each adjoined slice $x^{d-1} y$ has border rank two.
\end{remark}

We conclude with some open problems.
\begin{itemize}
\item For a symmetric tensor $\Tcal$, compare the decomposable rank $\drk_J \Tcal$ with the symmetric decomposable rank $\sdrk_J \Tcal$ across subsets $J \subset [d]$. The two ranks coincide for $|J| = 1$, since the slice space $\Lcal_J$ is then a linear space of vectors, but they may differ for $|J| = d-1$. 
\end{itemize}
It remains unknown whether counterexamples to Comon's conjecture exist for small tensors, and whether they exist at low ranks, see~\cite[Problem 5.5]{seigal2019structured}. We mention next steps for these lines of investigation.
\begin{itemize}
\item Find other symmetric tensors $\Ccal \in (\RR^I)^{\otimes d}$ and finite sets of symmetric tensors $\Mcal \subset (\RR^I)^{\otimes (d-1)}$ that satisfy Step 1 of the construction of a counterexample, i.e. for which there is strict inequality $\minrk (\Ccal \Mod \Mcal) < \minsrk (\Ccal \Mod \Mcal)$. Find an order three real example. The paper~\cite{shitov2018counterexample} gives an example over the complex numbers with $3 = \minrk (\Ccal \Mod \Mcal) < \minsrk (\Ccal \Mod \Mcal) = 4$. Find an example over the complex numbers with $\minrk (\Ccal \Mod \Mcal) = 1$.
\item Prove Conjecture~\ref{conj:modspan}, and its complex analogue~\cite[Conjecture 7]{shitov2018counterexample}, for a wider class of tensors, cf. Remark~\ref{rmk:conj_more}. 
\end{itemize}

\bigskip
{\bf Acknowledgements.}
We thank the anonymous referees for comments that improved the paper.
We thank JM Landsberg for helpful discussions. 
We are grateful for funding from an LMS Undergraduate Research Bursary (Grant ref. URB-2021-25). We thank Jared Tanner for supporting the project proposal.

\bibliographystyle{alpha}

\bibliography{references}

\newcommand{\etalchar}[1]{$^{#1}$}
\begin{thebibliography}{{GTE}15}

\bibitem[AAA{\etalchar{+}}21]{ahern2021blood}
David~J Ahern, Zhichao Ai, Mark Ainsworth, Chris Allan, et~al.
\newblock A blood atlas of {COVID}-19 defines hallmarks of disease severity and
  specificity.
\newblock {\em MedRxiv}, 2021.

\bibitem[AFT11]{alexeev2011tensor}
Boris Alexeev, Michael~A Forbes, and Jacob Tsimerman.
\newblock Tensor rank: Some lower and upper bounds.
\newblock In {\em 2011 IEEE 26th Annual Conference on Computational
  Complexity}, pages 283--291. IEEE, 2011.

\bibitem[AGH{\etalchar{+}}14]{anandkumar2014tensor}
Animashree Anandkumar, Rong Ge, Daniel Hsu, Sham~M Kakade, and Matus Telgarsky.
\newblock Tensor decompositions for learning latent variable models.
\newblock {\em Journal of machine learning research}, 15:2773--2832, 2014.

\bibitem[AGHK13]{anandkumar2013tensor}
Anima Anandkumar, Rong Ge, Daniel Hsu, and Sham~M Kakade.
\newblock A tensor approach to learning mixed membership community models,
  2013.

\bibitem[BCS13]{burgisser2013algebraic}
Peter B{\"u}rgisser, Michael Clausen, and Mohammad~A Shokrollahi.
\newblock {\em Algebraic complexity theory}, volume 315.
\newblock Springer Science \& Business Media, 2013.

\bibitem[BDE19]{bik2019polynomials}
Arthur Bik, Jan Draisma, and Rob~H Eggermont.
\newblock Polynomials and tensors of bounded strength.
\newblock {\em Communications in Contemporary Mathematics}, 21(07):1850062,
  2019.

\bibitem[BGL13]{buczynski2013determinantal}
Jaros{\l}aw Buczy{\'n}ski, Adam Ginensky, and Joseph~M Landsberg.
\newblock Determinantal equations for secant varieties and the
  {E}isenbud--{K}oh--{S}tillman conjecture.
\newblock {\em Journal of the London Mathematical Society}, 88(1):1--24, 2013.

\bibitem[BI11]{burgisser2011geometric}
Peter B{\"u}rgisser and Christian Ikenmeyer.
\newblock Geometric complexity theory and tensor rank.
\newblock In {\em Proceedings of the forty-third annual ACM symposium on Theory
  of computing}, pages 509--518, 2011.

\bibitem[BTY{\etalchar{+}}21]{bi2021tensors}
Xuan Bi, Xiwei Tang, Yubai Yuan, Yanqing Zhang, and Annie Qu.
\newblock Tensors in statistics.
\newblock {\em Annual review of statistics and its application}, 8:345--368,
  2021.

\bibitem[CCG12]{Carlini2012TS}
Enrico Carlini, Maria~Virginia Catalisano, and Anthony~V. Geramita.
\newblock The solution to the {W}aring problem for monomials and the sum of
  coprime monomials.
\newblock {\em Journal of Algebra}, 370:5--14, 2012.

\bibitem[CDM19]{calvi2019tight}
Giuseppe~G Calvi, Bruno~Scalzo Dees, and Danilo~P Mandic.
\newblock Tight lower bound on the tensor rank based on the maximally square
  unfolding.
\newblock {\em arXiv preprint arXiv:1909.05831}, 2019.

\bibitem[CGLM08]{comon2008symmetric}
Pierre Comon, Gene Golub, Lek-Heng Lim, and Bernard Mourrain.
\newblock Symmetric tensors and symmetric tensor rank.
\newblock {\em SIAM Journal on Matrix Analysis and Applications},
  30(3):1254--1279, 2008.

\bibitem[CHHN21]{cai2021mode}
HanQin Cai, Keaton Hamm, Longxiu Huang, and Deanna Needell.
\newblock Mode-wise tensor decompositions: Multi-dimensional generalizations of
  {CUR} decompositions.
\newblock {\em arXiv preprint arXiv:2103.11037}, 2021.

\bibitem[CKOV17]{carlini2017real}
Enrico Carlini, Mario Kummer, Alessandro Oneto, and Emanuele Ventura.
\newblock On the real rank of monomials.
\newblock {\em Mathematische Zeitschrift}, 286(1):571--577, 2017.

\bibitem[DL20]{domanov2020uniqueness}
Ignat Domanov and Lieven~De Lathauwer.
\newblock On uniqueness and computation of the decomposition of a tensor into
  multilinear rank-{$(1,L\_r,L\_r)$} terms.
\newblock {\em SIAM Journal on Matrix Analysis and Applications},
  41(2):747--803, 2020.

\bibitem[DSL08]{de2008tensor}
Vin De~Silva and Lek-Heng Lim.
\newblock Tensor rank and the ill-posedness of the best low-rank approximation
  problem.
\newblock {\em SIAM Journal on Matrix Analysis and Applications},
  30(3):1084--1127, 2008.

\bibitem[Fri16]{friedland2016remarks}
Shmuel Friedland.
\newblock Remarks on the symmetric rank of symmetric tensors.
\newblock {\em SIAM Journal on Matrix Analysis and Applications},
  37(1):320--337, 2016.

\bibitem[GOV19]{gesmundo2019partially}
Fulvio Gesmundo, Alessandro Oneto, and Emanuele Ventura.
\newblock Partially symmetric variants of {C}omon's problem via simultaneous
  rank.
\newblock {\em SIAM Journal on Matrix Analysis and Applications},
  40(4):1453--1477, 2019.

\bibitem[{GTE}15]{gtex2015genotype}
{GTEx Consortium}.
\newblock The genotype-tissue expression ({GTE}x) pilot analysis: multitissue
  gene regulation in humans.
\newblock {\em Science}, 348(6235):648--660, 2015.

\bibitem[Hac12]{hackbusch2012tensor}
Wolfgang Hackbusch.
\newblock {\em Tensor spaces and numerical tensor calculus}, volume~42.
\newblock Springer, 2012.

\bibitem[H{\aa}s89]{haastad1989tensor}
Johan H{\aa}stad.
\newblock Tensor rank is {NP}-complete.
\newblock In {\em International Colloquium on Automata, Languages, and
  Programming}, pages 451--460. Springer, 1989.

\bibitem[HL13]{hillar2013most}
Christopher~J Hillar and Lek-Heng Lim.
\newblock Most tensor problems are {NP}-hard.
\newblock {\em Journal of the ACM (JACM)}, 60(6):1--39, 2013.

\bibitem[HVB{\etalchar{+}}16]{hore2016tensor}
Victoria Hore, Ana Vinuela, Alfonso Buil, Julian Knight, Mark~I McCarthy,
  Kerrin Small, and Jonathan Marchini.
\newblock Tensor decomposition for multiple-tissue gene expression experiments.
\newblock {\em Nature genetics}, 48(9):1094--1100, 2016.

\bibitem[IK99]{iarrobino1999power}
A.~Iarrobino and V.~Kanev.
\newblock {\em Power Sums, Gorenstein Algebras, and Determinantal Loci}.
\newblock Lecture Notes in Mathematics. Springer Berlin Heidelberg, 1999.

\bibitem[KB06]{kolda2006matlab}
Tamara~G Kolda and Brett~W Bader.
\newblock Matlab tensor toolbox.
\newblock Technical report, Sandia National Laboratories (SNL), Albuquerque,
  NM, and Livermore, CA., 2006.

\bibitem[Lan12]{landsberg2012tensors}
Joseph~M Landsberg.
\newblock Tensors: geometry and applications.
\newblock {\em Representation theory}, 381(402):3, 2012.

\bibitem[Lan17]{landsberg2017geometry}
Joseph~M Landsberg.
\newblock {\em Geometry and complexity theory}, volume 169.
\newblock Cambridge University Press, 2017.

\bibitem[LC09]{lim2009nonnegative}
Lek-Heng Lim and Pierre Comon.
\newblock Nonnegative approximations of nonnegative tensors.
\newblock {\em Journal of Chemometrics: A Journal of the Chemometrics Society},
  23(7-8):432--441, 2009.

\bibitem[LT10]{landsberg2010ranks}
Joseph~M Landsberg and Zach Teitler.
\newblock On the ranks and border ranks of symmetric tensors.
\newblock {\em Foundations of Computational Mathematics}, 10(3):339--366, 2010.

\bibitem[McC18]{mccullagh2018tensor}
Peter McCullagh.
\newblock {\em Tensor methods in statistics}.
\newblock Chapman and Hall/CRC, 2018.

\bibitem[MD09]{mahoney2009cur}
Michael~W Mahoney and Petros Drineas.
\newblock {CUR} matrix decompositions for improved data analysis.
\newblock {\em Proceedings of the National Academy of Sciences},
  106(3):697--702, 2009.

\bibitem[Rod21]{rodriguez2021rank}
Jorge~Tom{\'a}s Rodr{\'\i}guez.
\newblock On the rank and the approximation of symmetric tensors.
\newblock {\em Linear Algebra and its Applications}, 628:72--102, 2021.

\bibitem[RS19]{robeva2019duality}
Elina Robeva and Anna Seigal.
\newblock Duality of graphical models and tensor networks.
\newblock {\em Information and Inference: A Journal of the IMA}, 8(2):273--288,
  2019.

\bibitem[SBB{\etalchar{+}}20]{schurch2020coordinated}
Christian~M Sch{\"u}rch, Salil~S Bhate, Graham~L Barlow, Darci~J Phillips,
  et~al.
\newblock Coordinated cellular neighborhoods orchestrate antitumoral immunity
  at the colorectal cancer invasive front.
\newblock {\em Cell}, 182(5):1341--1359, 2020.

\bibitem[Sei19]{seigal2019structured}
Anna~Leah Seigal.
\newblock {\em Thesis. Structured tensors and the geometry of data}.
\newblock University of California, Berkeley, 2019.

\bibitem[Sei20]{seigal2020ranks}
Anna Seigal.
\newblock Ranks and symmetric ranks of cubic surfaces.
\newblock {\em Journal of Symbolic Computation}, 101:304--317, 2020.

\bibitem[Shi18]{shitov2018counterexample}
Yaroslav Shitov.
\newblock A counterexample to {C}omon's conjecture.
\newblock {\em SIAM Journal on Applied Algebra and Geometry}, 2(3):428--443,
  2018.

\bibitem[Shi19]{Shitov2019CounterexamplesTS}
Yaroslav Shitov.
\newblock Counterexamples to {S}trassen’s direct sum conjecture.
\newblock {\em Acta Mathematica}, 2019.

\bibitem[Shi20]{shitov2020comon}
Yaroslav Shitov.
\newblock Comon's conjecture over the reals.
\newblock {\em viXra preprint viXra:2009.0134}, 2020.

\bibitem[SNC{\etalchar{+}}17]{subramanian2017next}
Aravind Subramanian, Rajiv Narayan, Steven~M Corsello, David~D Peck, et~al.
\newblock A next generation connectivity map: L1000 platform and the first
  1,000,000 profiles.
\newblock {\em Cell}, 171(6):1437--1452, 2017.

\bibitem[Sul18]{sullivant2018algebraic}
Seth Sullivant.
\newblock {\em Algebraic statistics}, volume 194.
\newblock American Mathematical Soc., 2018.

\bibitem[VDDL16]{vervliet2016tensorlab}
Nico Vervliet, Otto Debals, and Lieven De~Lathauwer.
\newblock Tensorlab 3.0—numerical optimization strategies for large-scale
  constrained and coupled matrix/tensor factorization.
\newblock In {\em 2016 50th Asilomar Conference on Signals, Systems and
  Computers}. IEEE, 2016.

\bibitem[WDFS17]{wang2017operator}
Miaoyan Wang, Khanh~Dao Duc, Jonathan Fischer, and Yun~S Song.
\newblock Operator norm inequalities between tensor unfoldings on the partition
  lattice.
\newblock {\em Linear algebra and its applications}, 520:44--66, 2017.

\bibitem[ZHQ16]{zhang2016comon}
Xinzhen Zhang, Zheng-Hai Huang, and Liqun Qi.
\newblock Comon's conjecture, rank decomposition, and symmetric rank
  decomposition of symmetric tensors.
\newblock {\em SIAM Journal on Matrix Analysis and Applications},
  37(4):1719--1728, 2016.

\bibitem[ZHSX20]{zheng2020comon}
Baodong Zheng, Riguang Huang, Xiaoyu Song, and Jinli Xu.
\newblock On {C}omon's conjecture over arbitrary fields.
\newblock {\em Linear Algebra and its Applications}, 587:228--242, 2020.

\end{thebibliography}

\appendix

\section{Proofs from Section~\ref{sec:order6}.}
\label{sec:appendix} 

\begin{lemma}\label{lem:mu41}
The clone of $x^4 y$ is in $\Span\Wcal_1$, for $n \geq 5$.
\end{lemma}

\begin{proof}
Take the following linear combination of tensors in $\mathcal{W}_1 \subset (\RR^{E \cup \mathcal{E}})^{\otimes 5}$:
\begin{small}
\begin{align*}
 \sum_{\substack{1\leq i_1<i_2<i_3<i_4 \leq n \\ 1\leq k_1 \leq n}} (\al_{i_{1}}+\al_{i_{2}}+\al_{i_{3}}+\al_{i_{4}}|\al_{k_1})^{\otimes5}
+ \lambda_1 \sum_{\substack{1\leq i_1<i_2<i_3 \leq n \\ 1\leq k_1 \leq n}}\left(\al_{i_{1}}+\al_{i_{2}}+\al_{i_{3}}|\frac{n-3}{n-4}\al_{k_1}\right)^{\otimes5} \\
+ \lambda_2 \sum_{\substack{1\leq i_1<i_2 \leq n \\ 1\leq k_1 \leq n}}\left(\al_{i_{1}}+\al_{i_{2}}|\frac{n-2}{n-4}\al_{k_1}\right)^{\otimes 5}
+ \lambda_3 \sum_{\substack{1\leq i_1 \leq n \\ 1\leq k_1 \leq n}}\left(\al_{i_{1}}|\frac{n-1}{n-4}\al_{k_1}\right)^{\otimes5},\end{align*}
\end{small}
where $\lambda_1 = -
\frac{(n-4)^2}{n-3}$, $\lambda_2 = \frac{(n-3)(n-4)^2}{2(n-2)}$ and $\lambda_3 = -\frac{(n-2)(n-3)(n-4)^2}{6(n-1)}$.
This $\Tcal$ coincides with the clone of $x^4 y$, on all entries except its diagonal blocks $\Tcal_E$ and $\Tcal_\mathcal{E}$.
We correct the diagonal blocks by adding the following linear combination of tensors in $\Wcal_1$:
\begin{align*}
    \lambda_4 \sum\limits_{1\leq i_1<i_2<i_3<i_4 \leq n}(\al_{i_{1}}+\al_{i_{2}}+\al_{i_{3}}+\al_{i_{4}}|0)^{\otimes 5}+
\lambda_5 \sum\limits_{1\leq i_1<i_2<i_3 \leq n}(\al_{i_{1}}+\al_{i_{2}}+\al_{i_{3}}|0)^{\otimes5} \\ + \lambda_6 \sum\limits_{1\leq i_1<i_2 \leq n}(\al_{i_{1}}+\al_{i_{2}}|0)^{\otimes5}
+\lambda_7 \sum\limits_{1\leq i_1 \leq n}(\al_{i_1}|0)^{\otimes5}
+ \lambda_8 \sum\limits_{1\leq k \leq n}(0|\al_{k})^{\otimes5},
\end{align*}
where $\lambda_4 = -n$, $\lambda_5 = \frac{(n-4)^2 n}{n-3}$, $\lambda_6 = -\frac{(n-3)(n-4)^2 n}{2(n-2)}$, $\lambda_7 = \frac{(n-2)(n-3)(n-4)^2 n}{6(n-1)}$, and $\lambda_8 = -(\tbinom{n}{4}-\frac{(n-3)^4}{(n-4)^3}\tbinom{n}{3}+\frac{(n-3)(n-2)^4}{2(n-4)^3}\tbinom{n}{2}
-\frac{(n-2)(n-3)(n-1)^4 n}{6(n-4)^3})$. 
\end{proof}

\begin{lemma}\label{lem:mu32}
The clone of $x^3 y^2$ is in $\Span\Wcal_1$, for $n \geq 5$.
\end{lemma}
\begin{proof}Take the following linear combination of tensors in $\mathcal{W}_1 \subset (\RR^{E \cup \mathcal{E}})^{\otimes 5}$:
\begin{small}
\begin{align*}
\sum_{\substack{1\leq i_{1}<i_{2}<i_{3} \leq n \\ 1\leq k_1 <k_2 \leq n}}(\al_{i_{1}}+\al_{i_{2}}+\al_{i_{3}}|\al_{k_1}+\al_{k_2})^{\of}
+ \mu_1 \sum_{\substack{1\leq i_{1}<i_{2} \leq n \\ 1\leq k_1 <k_2 \leq n}}\left(\al_{i_{1}}+\al_{i_{2}}|\frac{n-2}{n-3}(\al_{k_1}+\al_{k_2})\right)^{\of}
\\
+\mu_2 \sum_{\substack{1\leq i_1\leq n \\ 1\leq k_1 <k_2 \leq n}}\left(\al_{i_1}|\frac{n-1}{n-3}(\al_{k_1}+\al_{k_2})\right)^{\of}
+ \mu_3 \sum_{\substack{1\leq i_1<i_2<i_3 \leq n \\ 1\leq k_1 \leq n}}\left( \al_{i_{1}}+\al_{i_{2}}+\al_{i_{3}}|\frac{n-2}{n-1}\al_{k_1}\right)^{\of}
\\
+\mu_4 \sum_{\substack{1\leq i_1<i_2 \leq n \\ 1\leq k_1 \leq n}}\left(\al_{i_{1}}+\al_{i_{2}}|\frac{(n-2)^2}{(n-3)(n-1)}\al_{k_1}\right)^{\of}
+ \mu_5 \sum_{\substack{1\leq i\leq n \\ 1\leq k_1 \leq n}}\left(\al_i|\frac{n-2}{n-3}\al_{k_1}\right)^{\of},
\end{align*}
\end{small}
where $\mu_1 = -\frac{(n-3)^3}{(n-2)^2}$, $\mu_2 = \frac{(n-2)(n-3)^3}{2(n-1)^2}$, $\mu_3 = -\frac{(n-1)^2}{n-2}$, $\mu_4 =\frac{(n-1)^2(n-3)^3}{(n-2)^3}$, $\mu_5 = -\frac{(n-3)^3}{2}$.
This tensor $\Tcal$ agrees with the clone of $x^3 y^2$ on all except the blocks $\Tcal_E$ and $\Tcal_\mathcal{E}$. We fix these blocks by adding on the linear combination 
\begin{align*}
\mu_6 \sum\limits_{1\leq i_{1}<i_{2}<i_{3} \leq n}(\al_{i_{1}}+\al_{i_{2}}+\al_{i_{3}}|0)^{\of}
+ \mu_7 \sum\limits_{1\leq i_1<i_2\leq n}(\al_{i_1}+\al_{i_2}|0)^{\of}
\\ + \mu_8 \sum\limits_{1\leq i_1\leq n}(\al_{i_1}|0)^{\of}
+ \mu_9 \sum\limits_{1\leq k_1 \leq n}(0|\al_{k_1})^{\of}
+ \mu_{10} \sum\limits_{1\leq k_1 < k_2 \leq n}(0|\al_{k_1}+\al_{k_2})^{\of},
\end{align*}
where $\mu_6 = -\tbinom{n}{2}+\frac{(n-1)^2 n}{(n-2)}$, $\mu_7 = \frac{(n-3)^3}{(n-2)^2}\tbinom{n}{2}-\frac{(n-1)^2(n-3)^3 n}{(n-2)^3}$, $\mu_8 = -\frac{(n-2)(n-3)^3}{2(n-1)^2}\tbinom{n}{2} + \frac{(n-3)^3 n}{2}$, $\mu_9 = \frac{(n-2)^4}{(n-1)^3}\tbinom{n}{3}-\frac{(n-2)^7}{(n-1)^3(n-3)^2}\tbinom{n}{2}+\frac{(n-2)^5n}{2(n-3)^2}$, and $\mu_{10} = -\tbinom{n}{3}+\tbinom{n}{2}\frac{(n-2)^3}{(n-3)^2}-\frac{(n-1)^3(n-2)n}{2(n-3)^2}$.
\end{proof}

\bigskip

\bigskip

\noindent
\footnotesize {\bf Authors' addresses:}

\noindent Kexin Wang, University of Oxford, {\tt kexin.wang@queens.ox.ac.uk}.

\noindent  Anna Seigal, Harvard University, {\tt aseigal@seas.harvard.edu}.

\end{document}